\documentclass[11pt,a4paper,reqno]{amsart}
\usepackage{stmaryrd}
\usepackage[english]{babel}
\usepackage{pst-grad} 
\usepackage{pst-plot} 
\usepackage{pstricks}
\usepackage{amsmath,amssymb,mathrsfs,amsthm, mathtools}
\usepackage{tikz} 
\usepackage{mathcomp,wasysym}  
\usepackage{graphicx}  
\usepackage[all]{xy} \xyoption{arc} \xyoption{color}
\usepackage{epsfig}
\usepackage{cite}
\usepackage[a4paper,
left=2.5cm, right=2.5cm,
top=3cm, bottom=3cm]{geometry}

\newcommand{\pare}[1]{\left( #1 \right)}

\newcommand{\av}[1]{\left| #1 \right|}
\newcommand{\bra}[1]{\left[ #1 \right]}
\newcommand{\set}[1]{\left\{ #1 \right\}}

\newcommand{\dd}{\textnormal{d}}

\newcommand{\cH}{\mathcal{H}}

\newcommand{\bR}{\mathbb{R}}

\newcommand{\pax}{\partial_x}

\normalsize
\normalsize
\def\comm#1#2{{\left\llbracket#1,#2\right\rrbracket}}

\usepackage[bookmarks,colorlinks]{hyperref}

\newtheorem{satz}{Proposition}[section]

\newtheorem{lemma}[satz]{Lemma} 

\theoremstyle{definition}

\title{Models for damped water waves}

\author[R. Granero-Belinch\'{o}n]{Rafael Granero-Belinch\'{o}n}
\email{rafael.granero@unican.es}
\address{Departamento  de  Matem\'aticas,  Estad\'istica  y  Computaci\'on,  Universidad  de Cantabria.  Avda.  Los  Castros  s/n,  Santander,  Spain.}

\author[S. Scrobogna]{Stefano Scrobogna}
\email{sscrobogna@bcamath.org}
\address{Basque Center for Applied Mathematics, Mazarredo 14, 48009, Bilbao, Basque Country, Spain}

\begin{document}
\begin{abstract}
In this paper we derive some new weakly nonlinear asymptotic models describing viscous waves in deep water with or without surface tension effects. These asymptotic models take into account several different  dissipative effects and are obtained from the free boundary problems formulated in the works of Dias, Dyachenko and Zakharov (\emph{Physics Letters A, 2008}), Jiang, Ting, Perlin and Schultz (\emph{Journal of Fluid Mechanics,1996}) and Wu, Liu and Yue (\emph{Journal of Fluid Mechanics, 2006}). 
\end{abstract}



\maketitle
{\small
\tableofcontents}

\allowdisplaybreaks
\section{Introduction}
The motion of a free boundary irrotational and incompressible flow is a classical research topic \cite{Stokes_1847}. In most applications, the flow is also assumed to be inviscid \cite{Lannes13, Coutand-Shkoller:well-posedness-free-surface-incompressible}. 

However, even if in most situations in coastal engineering the assumption of inviscid flow leads to very accurate results, there are other physical scenarios where the viscosity needs to be taken into account. Moreover, there are many situations in which the viscosity is very large and the vorticity is small and  its effect negligible. Actually, certain discrepancies between experiments and inviscid theory have been previously reported in the literature. For instance, Wu \cite{wu1981long} found that

\begin{quote}\emph{From  this comparison the theory appears quite satisfactory in predicting the wave phases during the inward focusing and the subsequent reflection within a radial distance as far as $r = 20$, while the peak amplitudes observed in the experiments are slightly smaller than those predicted by the  theory. This  discrepancy can be ascribed to the neglect of the viscous effects in the theory and to the approximation that the initial  wave generated in the tank was not cylindrical in shape and departed slightly from a perfect solitary wave profile in the experiment.}
\end{quote}

In addition to this, Zabusky and Galvin \cite{zabusky1971shallow} wrote 

\begin{quote}\emph{A laboratory-data/numerical- solution comparison of the number of crests and troughs and their phases (or relative  locations within a period) shows only negligible  difference. As one expects, the crest-to-trough amplitudes differ somewhat more because they are more sensitive to dissipative forces. To quantify some of the details we recommend a study including dissipation.}
\end{quote}

and, furthermore, Longuet-Higgins \cite{longuet1992theory} stated that

\begin{quote}
\emph{For certain applications, however, viscous damping of the waves is important, and it would be highly convenient to have equations and boundary conditions of comparable simplicity as for undamped waves.}
\end{quote}

The purpose of this paper is to derive new weakly nonlinear asymptotic models (in the spirit of \cite{granero2018asymptotic, granero2019asymptotic, GS, CGSW18, matsuno1992nonlinear, matsuno1993two, matsuno1993nonlinear, ramani2019multiscale}) describing damped water waves and, at the same time, keeping the features of potential flows. We observe that, at fist sight, the idea of viscous damping of potential flows is somehow paradoxical since the hypothesis of irrotational velocity implies that the viscous term in the Navier-Stokes equations vanishes.

The problem of describing the motion of a irrotational, incompressible, inviscid and homogeneous fluid with a free surface in two dimensions is known as the 2D water waves problem. The equations for the water waves problem are \cite{zakharov1968stability}

\begin{subequations}\label{eq:1}
\begin{align}
\Delta \phi&=0&&\text{ in }\Omega(t)\times[0,T],\\
\rho\left(\phi_t+\frac{1}{2}|\nabla\phi|^2+ Gh\right)-\gamma\mathcal{K}&=0&&\text{ on }\Gamma(t)\times[0,T],\\
h_t&=\nabla\phi\cdot \left(1+\left(\partial_{1}h\right)^2\right)^{1/2}n&&\text{ on }\Gamma(t)\times[0,T],
\end{align}
\end{subequations}
where $G$ stands for the gravity force,
\begin{align}\label{Omegat}
\Omega(t) & = \set{ (x_1, x_2)\in\bR^2 \ \Big| \ {-L}\pi <x_1<{L}\pi\,, -\infty < x_2 < h(x_1,t)\,, \ t\in[0,T] }, \\
\label{Gammat}
\Gamma(t) & = \set{ \pare{x_1,h(x_1,t)}\in\bR^2 \ \Big| \ x_1\in\mathbb{S}^1\,,\ t\in[0,T] }
\end{align} 
are the the region occupied by the fluid and the surface wave, respectively. We write 
$$
n=(-\partial_1 h,1)/\left(1+\left(\partial_{x_1}h\right)^2\right)^{1/2}
$$ 
the unit normal to the surface wave, $2L\pi$ to denote the characteristic wavelength of the surface wave, $\phi$ for the scalar potential of the flow, \emph{i.e.} the velocity field $u$ satisfies $u=\nabla\phi$, $\rho$ is the density of the fluid, $\gamma$ for the surface tension coefficient and 
$$
\mathcal{K}=\frac{\partial_{1}^2h}{\left(1+\left(\partial_{1}h\right)^2\right)^{3/2}},
$$
is the curvature of the surface wave.

The first attempts to include viscosity effects go back as far as to the works of Boussinesq \cite{boussinesq1895lois} and Lamb \cite{lamb1932hydrodynamics}. Later on, Ruvinsky \& Freidman \cite{ruvinsky1987fine} formulated a system of equations for weakly damped surfaces waves in deep water and used this system to compute capillary-gravity ripples riding on the forward face of steep gravity waves (see also \cite{ruvinsky1985improvement}). Then, these first results were generalized by Ruvinsky, Feldstein \& Freidman \cite{ruvinsky1991numerical} and the following system is proposed
\begin{subequations}\label{eq:2}
\begin{align}
\Delta \phi&=0&& \text{ in }\Omega(t)\times[0,T],\\
\rho\left(\phi_t+\frac{1}{2}|\nabla\phi|^2+ Gh\right)-\gamma\mathcal{K}&=-2\mu \partial_2^2\phi &&\text{ on }\Gamma(t)\times[0,T],\\
h_t&=\nabla\phi\cdot \left(1+\left(\partial_{1}h\right)^2\right)^{1/2}n+v &&\text{ on }\Gamma(t)\times[0,T],\\
v_t&=\partial_{1}^2\partial_2\phi &&\text{ on }\Gamma(t)\times[0,T],
\end{align}
\end{subequations}
where $v$ and $\mu$ denote the vertical component of the vortex part of fluid velocity and the dynamic viscosity. Equation \eqref{eq:2} was also studied by Kharif, Skandrani \& Poitevin \cite{khariff1996frequency}.

Using a clever change of variables, Longuet-Higgins \cite{longuet1992theory} simplified the previous system and obtained
\begin{subequations}\label{eq:3}
\begin{align}
\Delta \phi&=0&&\text{ in }\Omega(t)\times[0,T],\\
\rho\left(\phi_t+\frac{1}{2}|\nabla\phi|^2+ Gh\right)-\gamma\mathcal{K}&=-4\mu \partial_n^2\phi&&\text{ on }\Gamma(t)\times[0,T],\\
h_t&=\nabla\phi\cdot \left(1+\left(\partial_{1}h\right)^2\right)^{1/2}n&&\text{ on }\Gamma(t)\times[0,T],
\end{align}
\end{subequations}

A similar model was also studied by Jiang, Ting, Perlin \& Schultz \cite{jiang1996moderate} and Wu, Liu \& Yue \cite{wu2006note}, namely
\begin{subequations}\label{eq:3b}
\begin{align}
\Delta \phi&=0&&\text{ in }\Omega(t)\times[0,T],\\
\rho\left(\phi_t+\frac{1}{2}|\nabla\phi|^2+ Gh\right)-\gamma\mathcal{K}&=-\delta \mathscr{D}^s \phi&&\text{ on }\Gamma(t)\times[0,T],\\
h_t&=\nabla\phi\cdot \left(1+\left(\partial_{1}h\right)^2\right)^{1/2}n&&\text{ on }\Gamma(t)\times[0,T],
\end{align}
\end{subequations}
where the dissipative terms are chosen as
\begin{equation}\label{D}
\mathscr{D}^2\phi=\partial_2^2\phi\text{ or }\mathscr{D}^0\phi=\phi.
\end{equation}
Another similar model where the dissipation acts only on the velocity is the one by Joseph \& Wang \cite[Equation (6.7) and (6.8)]{joseph2004dissipation} (see also Wang \& Joseph \cite{wang2006purely}).

We would like to remark that, in the models of damped water waves mentioned so far, there are no dissipative effects acting on the free surface.

In a more recent paper, Dias, Dyachenko \& Zakharov \cite{dias2008theory} proposed a system where the free surface experiments dissipative effects. In particular, based on the linear problem, these authors derived
\begin{subequations}\label{eq:4}
\begin{align}
\Delta \phi&=0&&\text{ in }\Omega(t)\times[0,T],\\
\rho\left(\phi_t+\frac{1}{2}|\nabla\phi|^2+ Gh\right)&=-2\mu \partial_2^2\phi&&\text{ on }\Gamma(t)\times[0,T],\\
h_t&=\nabla\phi\cdot \left(1+\left(\partial_{1}h\right)^2\right)^{1/2}n+2\frac{\mu}{\rho}\partial_{1}^2h &&\text{ on }\Gamma(t)\times[0,T],
\end{align}
\end{subequations}
as a model of viscous water waves. This model was also considered by several other authors. Dutykh \& Dias \cite{dutykh2007viscous} obtain a new set of viscous potential free-surface flow equations in the spirit of \eqref{eq:4} taking into account the effects of the bottom topography. These authors also derived a long wave approximation. This approximate model takes the form of a nonlocal (in time) Boussinesq system (see also \cite{dutykh2009visco, dutykh2007dissipative,dutykhA,dutykhB}). Kakleas \& Nicholls \cite{kakleas2010numerical}, using the analytic dependence of the Dirichlet-Neumann operator, derived a system of two equations modelling \eqref{eq:4}. These equations are the viscous analog of the classical Craig-Sulem WW2 model and were mathematically studied by Ambrose, Bona \& Nicholls \cite{ambrose2012well}. The well-posedness of the full \eqref{eq:4} was studied very recently by Ngom \& Nicholls \cite{ngom2018well}. In particular these authors proved global existence of solutions starting from a small enough initial data for the case of non-vanishing surface tension $\gamma\neq0$.

Some other related results are those by Kharif, Kraenkel, Manna \& Thomas \cite{kharif2010modulational} and Hunt \& Dutykh \cite{hunt2014visco}. Kharif, Kraenkel, Manna \& Thomas studied a similar situation to \eqref{eq:4} within the framework of a forced and damped nonlinear Schr\"odinger equation (see also Touboul \& Kharif \cite{touboul2010nonlinear}), while Hunt \& Dutykh considered the problem of the interface motion under the capillary-gravity and an external electric force in the case of an incompressible, viscous, perfectly conducting fluid. Finally, let us mention the recent work by Guyenne \& Parau \cite{GP} where the authors applied a simplified version  of \eqref{eq:4} to model wave attenuation in sea ice.

\subsection{Plan of the paper}
First we obtain the dimensionless Eulerian formulation in the moving domain and transform it to a dimensionless Arbitrary Lagrangian-Eulerian (ALE) formulation in a fixed domain in section \ref{ref:damped}. Then we introduce the asymptotic expansion and obtain the cascade of linear equations for the different scales presents in the problem with $s=0$ corresponding to the models by Jiang, Ting, Perlin \& Schultz \cite{jiang1996moderate} and Wu, Liu \& Yue \cite{wu2006note} in Section \ref{ref:asymdampeds0}. After neglecting errors of $\mathcal{O}(\varepsilon^2)$ we find the nonlocal wave equation modelling the case $s=0$. After that we consider the case $s=2$ and, following a similar approach, find the nonlocal wave equation for the model of Dias, Dyachenko, and Zakharov \cite{dias2008theory}. Finally, we conclude with a parabolic system of Craig-Sulem flavour in section \ref{ref:asymdampeds3}.

\subsection{Notation}
Let $A$ be a matrix, and $b$ be a column vector. Then, we write $A^i_j$ for the component of $A$, located on row $i$ and column $j$. We will use the Einstein summation convention for expressions with indexes.

We write
$$
\partial_j f=\frac{\partial f}{\partial x_j},\quad f_t=\frac{\partial f}{\partial t }
$$
for the space derivative in the $j-$th direction and for a time derivative, respectively. 
Unless parenthesis are involved, every differential operator acts locally. For instance,
$$
\partial_1f\partial_1 \eta=(\partial_1f)(\partial_1 \eta).
$$

Let $f(x_1)$ denote a $L^2$ function on $\mathbb{S}^1$ (as usually, identified with the interval $[-\pi,\pi]$ with periodic boundary conditions). We define the
Hilbert transform $\mathcal{H}$ and the Dirichlet-to-Neumann operator $ \Lambda $ and its powers,  respectively,  using Fourier series
\begin{align}\label{Hilbert}
\widehat{\mathcal{H}f}(k)=-i\text{sgn}(k) \hat{f}(k) \,, \ \ 
\widehat{\Lambda f}(k)=|k|\hat{f}(k)\,, \ \ \widehat{\Lambda^s f}(k)&=|k|^s\hat{f}(k)\,,
\end{align}
where 
$$
\hat f(k) = {\dfrac{1}{2\pi}} \int_{\mathbb{S}^1}{}\, f(x_1) \ 
e^{-ikx_1}dx_1.
$$
In particular, for zero-mean functions, we note that
$$
\partial_1 \mathcal{H}=\Lambda,\quad \mathcal{H}^2=-1,\quad \partial_1\Lambda^{-1} = -\cH.
$$
These last equalities will be used extensively through the whole text. Finally, we define the commutator as
$$
\comm{A}{B}f=A(Bf)-B(Af).  
$$

\section{Damped water waves}\label{ref:damped}
\subsection{The equations in the Eulerian formulation}
We consider system the system
\begin{subequations}\label{eq:all}
\begin{align}
\Delta \phi&=0&&\text{ in }\Omega(t)\times[0,T],\\
\rho\left(\phi_t+\frac{1}{2}|\nabla\phi|^2+ Gh\right)-\gamma\mathcal{K}&=-\delta_1 \mathscr{D}^s \phi&&\text{ on }\Gamma(t)\times[0,T],\\
h_t&=\nabla\phi\cdot \left(1+\left(\partial_{1}h\right)^2\right)^{1/2}n+\delta_2 \partial_{1}^2 h&&\text{ on }\Gamma(t)\times[0,T],
\end{align}
\end{subequations}
where $\delta_i\geq0$ are constant, the dissipative terms are as in \eqref{D}, $\phi$ is the scalar potential (units of $length^2/time$), $h$ denotes the surface wave (units of $length$) and $G$ (units of $length/time^2$) is the gravity acceleration. The constant $\delta_1$ has units of $mass/(length^2\cdot time)$ (when $\mathscr{D}^0\phi=\phi$)  and of $mass/time$ (when $\mathscr{D}^2\phi=\partial_2^2\phi$) while $\delta_2$ has units of $length^2/time$. We observe that, for appropriate choice of $\delta_i$ and $s$ we recover (exactly) \eqref{eq:3b} and \eqref{eq:4}. Indeed, if $\delta_2=0$ we obtain the same model by Jiang, Ting, Perlin \& Schultz \cite{jiang1996moderate} ($\delta_2=0$ and $s=0$) and Wu, Liu \& Yue \cite{wu2006note} ($delta_2=0$ and $s=0$ or $s=2$), while if $\delta_2=\delta_1/\rho$ and $s=2$ we recover the model by Dias, Dyachenko \& Zakharov \cite{dias2008theory}. 

Following the pioneer work of Zakharov \cite{zakharov1968stability}, we use the trace of the velocity potential $\xi(t,x)=\phi(t,x,h(t,x))$ (units of $length^2/time$). Now we observe that
\begin{align*}
\xi_t(t,x)&=\phi_t(t,x,h(t,x))+\partial_2\phi(t,x,h(t,x))h_t(t,x)\\
&=\phi_t(t,x,h(t,x))+\partial_2\phi(t,x,h(t,x))\left(\nabla\phi\cdot(-\pax h,1)+\delta_2 \partial_1^2 h\right).
\end{align*}
Thus, \eqref{eq:all} can be written as
\begin{subequations}\label{eq:all2}
\begin{align}
\Delta \phi&=0&&\text{ in }\Omega(t)\times[0,T],\\
\phi  &= \xi \qquad &&\text{on }\Gamma(t)\times[0,T],\\
\xi_t &=\partial_2\phi\left(\nabla\phi\cdot(-\partial_1 h,1)+\delta_2 \partial_1^2 h\right)-\frac{1}{2}|\nabla \phi|^2-Gh+\frac{\gamma}{\rho}\mathcal{K}-\frac{\delta_1}{\rho} \mathscr{D}^s \phi&&\text{ on }\Gamma(t)\times[0,T],\\
h_t&=\nabla\phi\cdot \left(1+\left(\partial_{1}h\right)^2\right)^{1/2}n+\delta_2 \partial_{1}^2 h&&\text{ on }\Gamma(t)\times[0,T].
\end{align}
\end{subequations}
The system (\ref{eq:all2}) is supplemented with an initial condition for $h$ and $\xi$:
\begin{align}\label{eq:initial}
h(x,0)&=h_0(x),\\
\xi(x,0)&=\phi(x,h(0,x),0)=\xi_0(x).
\end{align}

\subsection{Nondimensional Eulerian formulation} \label{sec:nondim_Darcy}
We denote by $H$ and $L$ the typical amplitude and wavelength of the water wave. We change to dimensionless variables (denoted with $\tilde{\cdot}$)
\begin{align}\label{dimensionlessA}
x=L \ \tilde{x}, && t=\sqrt{\frac{L}{G}}\ \tilde{t},
\end{align}
and unknowns
\begin{align}\label{dimensionlessB}
h(x_1,t)=H  \ \tilde{h}(\tilde{x}_1,\tilde{t}), && \phi(x_1,x_2,t)=H\sqrt{G L}\tilde{\phi}(\tilde{x}_1,\tilde{x}_2,\tilde{t}).
\end{align}
with the non-dimensionalized fluid domain
\begin{align*}
\widetilde{\Omega}(t) & =\set{ \pare{ \tilde{x}_1, \tilde{x}_2} \ \left| \ -\pi<\tilde{x}_1< \pi\,, -\infty < \tilde{x}_2 < \frac{H}{L}\tilde{h}(\tilde{x}_1,t)\,, \ t\in[0,T]\right.  }, \\
\widetilde{\Gamma}(t) & =\set{ \pare{ \tilde{x}_1, \frac{H}{L}\tilde{h}(\tilde{x}_1,t)}\,, \ t\in[0,T] }
\end{align*}
We find the following dimensionless parameters:
\begin{align} \label{eq:dimensionless_parameters}
\varepsilon=\frac{H}{L}, && \alpha_1^{s}=\frac{\delta_1}{\rho \sqrt{G}L^{s-\frac{1}{2}}}, && \alpha_2=\frac{\delta_2 }{\sqrt{G}L^{3/2}}. && \beta=\frac{\gamma}{\rho G L^2},
\end{align}
where $s=0$ if $\mathscr{D}\phi=\phi$ and $s=2$ if $\mathscr{D}\phi=\partial_2^2\phi$. The first parameter is known as the \emph{steepness parameter} and measures the ratio between the amplitude and the wavelength of the wave. The $\alpha's$ consider the ratio between gravity and viscosity forces. Finally, the fourth one is the Bond number that compares gravity forces with capillary forces. Dropping the tildes for the sake of clarity, we have the following dimensionless form of the damped water waves problem
\begin{subequations}\label{eq:all2dimensionless}
\begin{align}
\Delta \phi&=0&&\text{ in }\Omega(t)\times[0,T],\\
\phi  &= \xi \qquad &&\text{on }\Gamma(t)\times[0,T],\\
\xi_t &=-\frac{\varepsilon}{2}|\nabla \phi|^2- h+\frac{\beta\partial_{1}^2h}{\left(1+\left(\varepsilon\partial_{1}h\right)^2\right)^{3/2}}-\alpha_1^{s}\mathscr{D}^s \phi &&\nonumber\\
&\quad+\varepsilon\partial_2\phi\left(\nabla\phi\cdot(-\varepsilon\partial_1 h,1)+\alpha_2 \partial_1^2 h\right)&&\text{ on }\Gamma(t)\times[0,T],\\
h_t&=\nabla\phi\cdot (-\varepsilon\partial_1 h,1)+\alpha_2\partial_{1}^2 h&&\text{ on }\Gamma(t)\times[0,T],
\end{align}
\end{subequations}
where we have used the nondimensional parameters \eqref{eq:dimensionless_parameters}.

\subsection{The equations in the Arbitrary Lagrangian-Eulerian formulation}\label{sec:muskat_fixed domain}
In the present section we want to express system \eqref{eq:all2dimensionless} on the reference domain $\Omega$ and reference interface $\Gamma$
\begin{align}\label{Omega}
\Omega = \mathbb{S}^1 \times (-\infty, 0) \,, &&
\Gamma = \mathbb{S}^1 \times \{0\} \,.
\end{align} 
The easiest way to do so is, supposing that $ h $ is regular, defining the following family (parametrized in $ t\in \bra{0, T} $) of diffeomorphisms
\begin{equation*}
\begin{aligned}
&\psi : && \bra{0, T} \times\Omega && \to && \Omega\pare{t}, \\
&&& \pare{x_1, x_2, t} && \mapsto && \psi\pare{x_1, x_2,t} = \pare{x_1, x_2 + \varepsilon h\pare{x_1,t}}. 
\end{aligned}
\end{equation*}
Such a technique has been already used in the past by different authors (see for instance \cite{CGSW18, Coutand-Shkoller:well-posedness-free-surface-incompressible, granero2019asymptotic, Lannes13, ngom2018well} and the references therein). We compute 
\begin{align}\label{eq:diif_matrices}
\nabla \psi = \pare{
\begin{array}{cc}
1 & 0 \\
\varepsilon \partial_1 h\pare{x_1,t} & 1
\end{array}
}, &&
A = \pare{\nabla\psi}^{-1} = \pare{
\begin{array}{cc}
1 & 0 \\
-\varepsilon \partial_1 h\pare{x_1,t} & 1
\end{array}
}. 
\end{align}
With such map we can define the push-back of any application $ \theta $ defined on $ \Omega\pare{t} $ simply as $ \Theta = \theta \circ \psi $, whence in particular we define
\begin{align*}
\Phi = \phi \circ \psi\;. 
\end{align*}
We let $N =e_2$ denote the outward unit normal to $\Omega$ at $\Gamma$. We also recall that, if $ \Theta = \theta \circ \psi $, the following formula holds
\begin{equation*}
\partial_j \theta \circ \psi = A^k_j \partial_k \Theta, 
\end{equation*}
where Einstein convention is used. Then, we can rewrite \eqref{eq:all2dimensionless} as the following system of variable coefficients nonlinear PDEs posed on a fixed reference domain
\begin{subequations}\label{eq:ALE}
\begin{align}
A^\ell_j\partial_\ell\left(A^k_j\partial_k \Phi\right)&=0&&\text{ in }\Omega\times[0,T],\\
\Phi  &= \xi \qquad &&\text{on }\Gamma\times[0,T],\\
\xi_t &=-\frac{\varepsilon}{2}A^k_j\partial_k\Phi A^\ell_j\partial_\ell\Phi- h+\frac{\beta\partial_{1}^2h}{\left(1+\left(\varepsilon\partial_{1}h\right)^2\right)^{3/2}}-\alpha_1^{s}\mathcal{D}^s \Phi &&\nonumber\\
&\quad+\varepsilon A^k_2\partial_k\Phi\left(A^\ell_j\partial_\ell\Phi A^2_j+\alpha_2 \partial_1^2 h\right)&&\text{ on }\Gamma\times[0,T],\\
h_t&=A^k_j\partial_k\Phi A^2_j+\alpha_2\partial_{1}^2 h&&\text{ on }\Gamma\times[0,T],
\end{align}
\end{subequations}
where the operator $\mathcal{D}^s$ is
$$
\mathcal{D}^0\Phi=\xi,\;\mathcal{D}^2\Phi=A^\ell_2\partial_\ell\left(A^k_2\partial_k \Phi\right).
$$

Next we explicit the values of the $ A_j^i $'s in the above system (see \eqref{eq:diif_matrices}) obtaining hence 
\begin{subequations}\label{eq:ALE2}
\begin{align}
\Delta \Phi &={\varepsilon\pare{\partial_1^2 h \ \partial_2 \Phi + 2\partial_1 h \ \partial_{12}\Phi}-\varepsilon^2(\partial_1 h)^2\partial^2_2\Phi}  ,  &&\text{in } \Omega\times[0,T]\,,\\
\Phi  &= \xi \qquad &&\text{on }\Gamma\times[0,T],\\
\xi_t &=-\frac{\varepsilon}{2}\left[(\partial_1\Phi)^2+(\varepsilon \partial_1 h\partial_2\Phi)^2+(\partial_2\Phi)^2-2\varepsilon \partial_1 h\partial_2\Phi\partial_1\Phi\right]&&\nonumber\\
&\quad- h+\frac{\beta\partial_{1}^2h}{\left(1+\left(\varepsilon\partial_{1}h\right)^2\right)^{3/2}}-\alpha_1^{s}\mathcal{D}^s \Phi &&\nonumber\\
&\quad+\varepsilon \partial_2\Phi\left(-\varepsilon \partial_1 h \partial_1 \Phi+\varepsilon^2 (\partial_1 h)^2 \partial_2\Phi +\partial_2\Phi+\alpha_2 \partial_1^2 h\right)&&\text{ on }\Gamma\times[0,T],\\
h_t&=-\varepsilon \partial_1 h \partial_1 \Phi+\varepsilon^2 (\partial_1 h)^2 \partial_2\Phi +\partial_2\Phi+\alpha_2\partial_{1}^2 h&&\text{ on }\Gamma\times[0,T],
\end{align}
\end{subequations}
where
$$
\mathcal{D}^0\Phi=\Phi,\;\mathcal{D}^2\Phi=\partial_2^2 \Phi.
$$

\section{The asymptotic model for damped water waves when $s=0$} \label{ref:asymdampeds0}
In this section we consider the case $s=0$ (the model by Jiang, Ting, Perlin \& Schultz \cite{jiang1996moderate} and Wu, Liu \& Yue \cite{wu2006note}). In this case we have that
$$
\mathcal{D}^0\Phi=\Phi.
$$
We introduce the following ansatz:
\begin{equation}
\label{eq:ansatz}
\begin{aligned}
\Phi\pare{x_1,x_2,t} & = \sum_n \varepsilon^n \Phi^{\pare{n}}\pare{x_1,x_2,t}, \\
\xi\pare{x_1,t} & = \sum_n \varepsilon^n \xi^{\pare{n}}\pare{x_1,t}, \\
h\pare{x_1,t} & = \sum_n \varepsilon^n h^{\pare{n}}\pare{x_1,t}.
\end{aligned}
\end{equation} 
With this ansatz we can re-profile the nonlinear system \eqref{eq:ALE2} in an equivalent sequence of linear systems where the evolution of the $ n $-th profile is determined by the evolution of the preceding $ n-1 $ profiles. 

We are interested in a model approximating \eqref{eq:ALE2} with an error $\mathcal{O}(\varepsilon^2)$. Using that
$$
\frac{1}{\left(1+x^2\right)^{3/2}}=1+\mathcal{O}(x^2),
$$
we obtain that
$$
\frac{\beta\partial_{1}^2h}{\left(1+\left(\varepsilon\partial_{1}h\right)^2\right)^{3/2}}=\beta\partial_{1}^2h+\mathcal{O}(\varepsilon^2).
$$
For the case $n=0$, we have that
\begin{subequations}\label{eq:n0}
\begin{align}
\Delta \Phi^{(0)} &=0 ,  &&\text{in } \Omega\times[0,T]\,,\\
\Phi^{(0)}  &= \xi^{(0)}  \qquad &&\text{on }\Gamma\times[0,T],\\
\xi^{(0)} _t &=- h^{\pare{0}} +\beta\partial_{1}^2h^{\pare{0}} -\alpha_1^{0}\Phi^{(0)}  &&\text{ on }\Gamma\times[0,T],\\
h_t^{(0)} &=\partial_2\Phi^{(0)} +\alpha_2\partial_{1}^2 h^{\pare{0}} &&\text{ on }\Gamma\times[0,T].
\end{align}
\end{subequations}
Recalling that
$$
\widehat{\Phi^{\pare{0}}} \pare{k, x_2, t}  = \xi^{(0)}(k,t)e^{|k|x_2}\qquad \text{in } \Omega\times[0,T]
$$
so
$$
\partial_2 \Phi^{(0)}=\Lambda \xi^{(0)}\qquad \text{on }\Gamma,
$$
we find that (\ref{eq:n0}d) can be equivalently written as
\begin{align*}
h_{tt}^{(0)} &=\Lambda\left(- h^{\pare{0}} +\beta\partial_{1}^2h^{\pare{0}} -\alpha_1^{0}\xi^{(0)}\right) +\alpha_2\partial_{1}^2 h^{\pare{0}}_t &&\text{ on }\Gamma\times[0,T].
\end{align*}
We note that (\ref{eq:n0}d) can be equivalently wwritten as
\begin{equation}
\label{eq:xi0_as_function_of_h0}
\xi^{(0)} =\Lambda^{-1}\left[h^{\pare{0}} _t -\alpha_2\partial_{1}^2 h^{\pare{0}}\right],
\end{equation}
thus,
\begin{align}\label{eq:eq_h0_0}
h_{tt}^{(0)} &=\Lambda\left(- h^{\pare{0}} +\beta\partial_{1}^2h^{\pare{0}} -\alpha_1^{0}\Lambda^{-1}\left[h^{\pare{0}} _t -\alpha_2\partial_{1}^2 h^{\pare{0}}\right]\right) +\alpha_2\partial_{1}^2 h^{\pare{0}}_t &&\text{ on }\Gamma\times[0,T].
\end{align}

Similarly, in the case $n=1$, we find that
\begin{subequations}\label{eq:n1}
\begin{align}
\Delta \Phi^{(1)} &=\partial_1^2 h^{\pare{0}} \ \partial_2 \Phi^{(0)} + 2\partial_1 h^{\pare{0}} \ \partial_{12}\Phi^{(0)}  ,  &&\text{in } \Omega\times[0,T]\,,\label{eq:n1Phi}\\
\Phi^{(1)}  &= \xi^{(1)} \qquad &&\text{on }\Gamma\times[0,T],\\
\xi^{(1)}_t &=\frac{1}{2}\left[(\partial_2\Phi^{(0)})^2-(\partial_1\Phi^{(0)})^2\right]&&\nonumber\\
&\quad- h^{\pare{1}}+\beta\partial_{1}^2h^{\pare{1}}-\alpha_1^{0}\Phi^{(1)} +\alpha_2\partial_2\Phi^{(0)} \partial_1^2 h^{\pare{0}}&&\text{ on }\Gamma\times[0,T],\label{eq:n1xi}\\
h^{\pare{1}}_t&=-\partial_1 h^{\pare{0}} \partial_1 \Phi^{(0)} +\partial_2\Phi^{(1)}+\alpha_2\partial_{1}^2 h^{\pare{1}}&&\text{ on }\Gamma\times[0,T],\label{eq:n1h}
\end{align}
\end{subequations}

Let us define
\begin{equation*}
b = \partial_1^2 h^{\pare{0}} \partial_2 \Phi ^{\pare{0}} + 2 \partial_1 h^{\pare{0}} \partial_{12} \Phi ^{\pare{0}}, 
\end{equation*}
We now use Lemma \ref{lem:solutions_Poisson} in order to compute
\begin{align*}
\partial_2 \widehat{\Phi^{\pare{1}}} \pare{k, 0, t} & = \int _{ -\infty}^0 \hat{b}\pare{k, y_2, t} e^{\av{k} y_2 } \textnormal{d} y_2 + \av{k} \widehat{\xi^{\pare{1}}}\pare{k, t}. 
\end{align*}

We want to provide an explicit expression for the term $ \int _{ -\infty}^0 \hat{b}\pare{k, y_2, t} e^{\av{k} y_2 } \textnormal{d} y_2  $ considering the form of $ b $. We compute that
\begin{align*}
\int _{ -\infty}^0 \hat{b}\pare{k, y_2, t} e^{\av{k} y_2 } \textnormal{d} y_2 & = - 
\int_{-\infty}^0 e^{\pare{\av{k}+ \av{m}}y_2}\pare{k-m} \pare{ k+m}\av{m}  \widehat{h^{(0)}}\pare{k-m}\widehat{\xi^{(0)}}\pare{m} \dd y_2
, \\
& = -
\frac{ \av{m}\bra{\av{k}^2 - \av{m}^2}}{\av{k} + \av{m}} 	\ \widehat{h^{(0)}}\pare{k-m}\widehat{\xi^{(0)}}\pare{m}
, \\
& = -
 \av{m}\bra{\av{k} - \av{m}} 	\ \widehat{h^{(0)}}\pare{k-m}\widehat{\xi^{(0)}}\pare{m}
.
\end{align*}

Thus, we find that
\begin{equation}\label{eq:explicit_expression_Phi1}
\left. \partial_2 \Phi^{\pare{1}}\right|_{x_2=0} =   \Lambda \xi^{\pare{1}} -\comm{\Lambda}{h^{\pare{0}}}\Lambda \xi^{ \pare{0}}.
\end{equation}

The evolution equations for $ h^{\pare{1}} $ and $ \xi^{\pare{1}} $ become hence
\begin{align}
h^{\pare{1}}_t & =-\partial_1 h^{\pare{0}} \partial_1 \xi^{(0)} +   \Lambda \xi^{\pare{1}} -\comm{\Lambda}{h^{\pare{0}}}\Lambda \xi^{ \pare{0}}+\alpha_2\partial_{1}^2 h^{\pare{1}}, \label{eq:n1h1} \\
\xi^{(1)}_t &=\frac{1}{2}\left[ \pare{ \Lambda\xi^{(0)}}^2 - \pare{\partial_1\xi^{(0)}}^2\right]  \nonumber  \\
&\quad- h^{\pare{1}}+\beta\partial_{1}^2h^{\pare{1}}-\alpha_1^{0}\Phi^{(1)}     
+\alpha_2 \Lambda\xi^{(0)}\partial_1^2 h^{\pare{0}}. \label{eq:n1xi1}
\end{align}

Using the above equation for $h_t^{(1)}$ \eqref{eq:n1h1} we can express $ \xi^{\pare{1}} $ as a function of $ h^{\pare{0}}$, $ \xi^{\pare{0}} $ and $h^{\pare{1}} $ as follows
\begin{equation}\label{eq:xi1}
\xi^{\pare{1}} = \Lambda^{-1}\bra{h^{\pare{1}}_t +\partial_1 h^{\pare{0}} \partial_1\xi^{\pare{0}} + \comm{\Lambda}{h^{\pare{0}}}\Lambda \xi^{ \pare{0}} - \alpha_2\partial_{1}^2 h^{\pare{1}} }.
\end{equation}
Time differentiating \eqref{eq:n1h1} and inserting \eqref{eq:n1xi1}, we deduce
\begin{multline*}
h^{\pare{1}}_{tt} =-\partial_1 h^{\pare{0}}_t \partial_1\xi^{\pare{0}}-\partial_1 h^{\pare{0}} \partial_1\xi^{\pare{0}}_t
 +\frac{1}{2} \Lambda \left[\pare{\Lambda \xi^{(0)}}^2- \left(\partial_1\xi^{(0)}\right)^2\right] \\
- \Lambda  h^{\pare{1}}+\beta\Lambda \partial_{1}^2h^{\pare{1}}-\alpha_1^{0}\Lambda \Phi^{(1)} +\alpha_2 \Lambda\pare{\Lambda \xi^{\pare{0}} \partial_{1}^2 h^{\pare{0}}}\\
  -\comm{\Lambda}{h^{\pare{0}}_t}\Lambda \xi^{ \pare{0}} - \comm{\Lambda}{h^{\pare{0}}}\Lambda \xi^{ \pare{0}}_t
 +\alpha_2\partial_{1}^2 h^{\pare{1}}_t. 
\end{multline*}
Recalling the definition of the Riesz potential $\Lambda^{-1}$ and using (\ref{eq:n0}c) and \eqref{eq:xi0_as_function_of_h0} in order to express $ \xi^{\pare{0}} $ and $ \xi^{\pare{0}}_t $ in terms of $ h^{\pare{0}}$, we find that 
\begin{multline*}
h^{\pare{1}}_{tt} = \partial_1 h^{\pare{0}}_t \cH \left[h^{\pare{0}} _t -\alpha_2\partial_{1}^2 h^{\pare{0}}\right]-\partial_1 h^{\pare{0}} \partial_1\bra{- h^{\pare{0}} +\beta\partial_{1}^2h^{\pare{0}} -\alpha_1^{0} \Phi^{(0)}}\\
  +\frac{1}{2}\Lambda \set{\left[h^{\pare{0}} _t -\alpha_2\partial_{1}^2 h^{\pare{0}}\right]^2 -\pare{ \cH\left[h^{\pare{0}} _t -\alpha_2\partial_{1}^2 h^{\pare{0}}\right]}^2} \\
-\Lambda  h^{\pare{1}}+\beta\Lambda \partial_{1}^2h^{\pare{1}}-\alpha_1^{0} \Lambda \Phi^{(1)} 
+\alpha_2 \Lambda \bra{ \pare{ h^{\pare{0}} _t -\alpha_2\partial_{1}^2 h^{\pare{0}}} \partial_1 ^2 h^{\pare{0}}  } \\
 -\comm{\Lambda}{h^{\pare{0}}_t} \left(h^{\pare{0}} _t -\alpha_2\partial_{1}^2 h^{\pare{0}}\right) - \comm{\Lambda}{h^{\pare{0}}}\Lambda \pare{- h^{\pare{0}} +\beta\partial_{1}^2h^{\pare{0}} -\alpha_1^{0} \Phi^{(0)}}
 +\alpha_2\partial_{1}^2 h^{\pare{1}}_t.
\end{multline*}
Using Tricomi identity 
\begin{equation}\label{tricomi}
(\cH f)^2-f^2=2\cH\left(f\cH f\right),
\end{equation}
the previous equation can be further simplified and we find that
\begin{multline*}
h^{\pare{1}}_{tt} = \partial_1 h^{\pare{0}}_t \cH \left[h^{\pare{0}} _t -\alpha_2\partial_{1}^2 h^{\pare{0}}\right]-\partial_1 h^{\pare{0}} \partial_1\bra{- h^{\pare{0}} +\beta\partial_{1}^2h^{\pare{0}} -\alpha_1^{0} \Phi^{(0)}}\\
  +\partial_1\set{\left[h^{\pare{0}} _t -\alpha_2\partial_{1}^2 h^{\pare{0}}\right]\cH\left[h^{\pare{0}} _t -\alpha_2\partial_{1}^2 h^{\pare{0}}\right]} \\
-\Lambda  h^{\pare{1}}+\beta\Lambda \partial_{1}^2h^{\pare{1}}-\alpha_1^{0} \Lambda \Phi^{(1)} 
+\alpha_2 \Lambda \bra{ \pare{ h^{\pare{0}} _t -\alpha_2\partial_{1}^2 h^{\pare{0}}} \partial_1 ^2 h^{\pare{0}}  } \\
 -\comm{\Lambda}{h^{\pare{0}}_t} \left(h^{\pare{0}} _t -\alpha_2\partial_{1}^2 h^{\pare{0}}\right) - \comm{\Lambda}{h^{\pare{0}}}\Lambda \pare{- h^{\pare{0}} +\beta\partial_{1}^2h^{\pare{0}} -\alpha_1^{0} \Phi^{(0)}}
 +\alpha_2\partial_{1}^2 h^{\pare{1}}_t.
\end{multline*}

We can express $ \alpha_1^0 \Phi^{\pare{0}} $ in terms of $ h^{\pare{0}} $ as follows
$$
\alpha_1^0 \Phi^{\pare{0}}\bigg{|}_{x_2=0}=\alpha_1^0 \xi^{\pare{0}}  = \alpha_1^0\Lambda^{-1}\left[h^{\pare{0}} _t -\alpha_2\partial_{1}^2 h^{\pare{0}}\right],
$$
and, inserting the previous formula into \eqref{eq:xi1}, we find that
\begin{equation*}
\begin{aligned}
\left. \alpha_1^0 \Phi^{\pare{1}} \right|_{x_2=0} & = \alpha_1^0  \xi^{\pare{1}}, \\
& = \alpha_1^0 \Lambda^{-1} \bra{h^{\pare{1}}_t +\partial_1 h^{\pare{0}} \partial_1\xi^{\pare{0}} + \comm{\Lambda}{h^{\pare{0}}}\Lambda \xi^{ \pare{0}} - \alpha_2\partial_{1}^2 h^{\pare{1}} } , \\
& =  \alpha_1^0 \Lambda^{-1} \set{ h^{\pare{1}}_t -\partial_1 h^{\pare{0}} \cH\bra{h^{\pare{0}} _t -\alpha_2\partial_{1}^2 h^{\pare{0}}} + \comm{\Lambda}{h^{\pare{0}}} \bra{h^{\pare{0}} _t -\alpha_2\partial_{1}^2 h^{\pare{0}}} - \alpha_2\partial_{1}^2 h^{\pare{1}} }.
\end{aligned} 
\end{equation*}

Substituting the previous expressions into the equation for $h_{tt}^{(1)}$, we deduce the following equation:
\begin{multline*}
h^{\pare{1}}_{tt} = \partial_1 h^{\pare{0}}_t \cH \left[h^{\pare{0}} _t -\alpha_2\partial_{1}^2 h^{\pare{0}}\right]-\partial_1 h^{\pare{0}} \partial_1\bra{- h^{\pare{0}} +\beta\partial_{1}^2h^{\pare{0}} -\alpha_1^0\Lambda^{-1}\left[h^{\pare{0}} _t -\alpha_2\partial_{1}^2 h^{\pare{0}}\right]}\\
  +\partial_1\set{\left[h^{\pare{0}} _t -\alpha_2\partial_{1}^2 h^{\pare{0}}\right]\cH\left[h^{\pare{0}} _t -\alpha_2\partial_{1}^2 h^{\pare{0}}\right]} -\Lambda  h^{\pare{1}}+\beta\Lambda \partial_{1}^2h^{\pare{1}}\\
-\alpha_1^0 \set{ h^{\pare{1}}_t -\partial_1 h^{\pare{0}} \cH\bra{h^{\pare{0}} _t -\alpha_2\partial_{1}^2 h^{\pare{0}}} + \comm{\Lambda}{h^{\pare{0}}} \bra{h^{\pare{0}} _t -\alpha_2\partial_{1}^2 h^{\pare{0}}} - \alpha_2\partial_{1}^2 h^{\pare{1}} }\\
+\alpha_2 \Lambda \bra{ \pare{ h^{\pare{0}} _t -\alpha_2\partial_{1}^2 h^{\pare{0}}} \partial_1 ^2 h^{\pare{0}}  }  -\comm{\Lambda}{h^{\pare{0}}_t} \left(h^{\pare{0}} _t -\alpha_2\partial_{1}^2 h^{\pare{0}}\right)\\
 - \comm{\Lambda}{h^{\pare{0}}} \pare{- \Lambda h^{\pare{0}} +\beta\Lambda\partial_{1}^2h^{\pare{0}} -\alpha_1^0\left[h^{\pare{0}} _t -\alpha_2\partial_{1}^2 h^{\pare{0}}\right]}
 +\alpha_2\partial_{1}^2 h^{\pare{1}}_t.
\end{multline*}

We group the nonlinear terms according to the coefficient in front. At $\mathcal{O}(1)$ we find that
\begin{multline}
\partial_1 h^{\pare{0}}_t \cH h^{\pare{0}} _t +\left(\partial_1 h^{\pare{0}}\right)^2+\frac{\Lambda}{2}\set{\left[h^{\pare{0}} _t\right]^2 -\pare{ \cH h^{\pare{0}} _t }^2}  -\comm{\Lambda}{h^{\pare{0}}_t} h^{\pare{0}} _t + \comm{\Lambda}{h^{\pare{0}}}\Lambda h^{\pare{0}}  \\
=-\Lambda\left(\left(\cH h^{(0)}_t\right)^2\right)+\partial_1\comm{\cH}{h^{(0)}}\Lambda h^{(0)},\label{O(1)}
\end{multline}
where we have used the identity \eqref{tricomi}. At $\mathcal{O}(\beta)$ we obtain that
\begin{align}
\beta\left(-\partial_1 h^{\pare{0}} \partial_{1}^3h^{\pare{0}} - \comm{\Lambda}{h^{\pare{0}}}\Lambda \partial_{1}^2h^{\pare{0}}\right)&=\beta\left(\Lambda\left(h^{\pare{0}}\Lambda^3 h^{\pare{0}}\right)-\partial_1\left(h^{\pare{0}}\partial_1^3 h^{\pare{0}}\right)\right)\nonumber\\
&=\beta\partial_1\comm{\cH}{h^{(0)}}\Lambda^3 h^{\pare{0}}.\label{O(beta)} 
\end{align}

At $\mathcal{O}(\alpha_2)$ we find the following contribution
\begin{multline}\label{eq:Oalpha_2v1}
-\alpha_2\bigg{[}\partial_1 h^{\pare{0}}_t \cH \partial_{1}^2 h^{\pare{0}} +\partial_1\set{h^{\pare{0}} _t \cH \partial_{1}^2 h^{\pare{0}}} +  \partial_1\set{\partial_{1}^2 h^{\pare{0}}\cH h^{\pare{0}} _t }\\
- \Lambda \bra{ h^{\pare{0}} _t  \partial_1 ^2 h^{\pare{0}}  }  -\comm{\Lambda}{h^{\pare{0}}_t}\partial_{1}^2 h^{\pare{0}}\bigg{]}\\
=-\alpha_2\bigg{[}\partial_1 h^{\pare{0}}_t \cH \partial_{1}^2 h^{\pare{0}} +\partial_1\set{h^{\pare{0}} _t \cH \partial_{1}^2 h^{\pare{0}}} +  \partial_1\set{\partial_{1}^2 h^{\pare{0}}\cH h^{\pare{0}} _t }\\
- 2\Lambda \bra{ h^{\pare{0}} _t  \partial_1 ^2 h^{\pare{0}}  }  +h^{\pare{0}}_t\Lambda\partial_{1}^2 h^{\pare{0}}\bigg{]}.
\end{multline}
Using
\begin{equation}\label{tricomi2}
\cH f\cH g-\cH\left(f\cH g+g\cH f\right)=fg,
\end{equation}
we find that
$$
\Lambda \bra{ h^{\pare{0}} _t  \partial_1 ^2 h^{\pare{0}}}=\Lambda\left(\cH h^{\pare{0}} _t\cH \partial_1 ^2 h^{\pare{0}}\right)+\partial_1\left(h^{\pare{0}} _t\cH \partial_1 ^2 h^{\pare{0}}+\partial_1 ^2 h^{\pare{0}}\cH h^{\pare{0}} _t\right).
$$
Thus, we can group terms in \eqref{eq:Oalpha_2v1} as follows
\begin{multline}\label{eq:Oalpha_2}
-\alpha_2\bigg{[}\partial_1 \left(h^{\pare{0}}_t \cH \partial_{1}^2 h^{\pare{0}}\right)- 2\Lambda\left(\cH h^{\pare{0}} _t\cH \partial_1 ^2 h^{\pare{0}}\right)-\partial_1\left(h^{\pare{0}} _t\cH \partial_1 ^2 h^{\pare{0}}+\partial_1 ^2 h^{\pare{0}}\cH h^{\pare{0}} _t\right) \bigg{]}\\
=\alpha_2\partial_1\comm{\cH}{\cH h^{\pare{0}} _t}\cH \partial_1 ^2 h^{\pare{0}}+\alpha_2\Lambda\left(\cH h^{\pare{0}} _t\cH \partial_1 ^2 h^{\pare{0}}\right).
\end{multline}
At $\mathcal{O}(\alpha_2\alpha_2)$, we find that
\begin{equation}\label{eq:Oalpha_22}
\alpha_2^2\bigg{[}\partial_1\set{\partial_{1}^2 h^{\pare{0}}\cH\partial_{1}^2 h^{\pare{0}}} - \Lambda\left[\left(\partial_{1}^2 h^{\pare{0}}\right)^2\right]  \bigg{]}=-\alpha_2^2\partial_1\comm{\cH}{\partial_{1}^2 h^{\pare{0}}}\partial_{1}^2 h^{\pare{0}}.
\end{equation}
We group now the $\mathcal{O}(\alpha_1^0)$ terms:
\begin{equation}\label{eq:Oalpha10}
\alpha_1^0\bigg{[}-\partial_1 h^{\pare{0}} \cH h^{\pare{0}} _t+\partial_1 h^{\pare{0}} \cH h^{\pare{0}} _t  - \comm{\Lambda}{h^{\pare{0}}} h^{\pare{0}} _t  + \comm{\Lambda}{h^{\pare{0}}} h^{\pare{0}} _t\bigg{]} =0.
\end{equation}
Finally, we are left with the $\mathcal{O}(\alpha_1^0\alpha_2)$ terms. These terms are
\begin{equation}\label{eq:Oalpha_01alpha2}
\alpha_1^0\alpha_2\bigg{[}\partial_1 h^{\pare{0}}\cH\partial_{1}^2 h^{\pare{0}}-\partial_1 h^{\pare{0}} \cH\partial_{1}^2 h^{\pare{0}} + \comm{\Lambda}{h^{\pare{0}}}\partial_{1}^2 h^{\pare{0}} 
 - \comm{\Lambda}{h^{\pare{0}}} \partial_{1}^2 h^{\pare{0}}\bigg{]}=0.
\end{equation}
Thus, using \eqref{O(1)}, \eqref{O(beta)}, \eqref{eq:Oalpha_2}, \eqref{eq:Oalpha_22}, \eqref{eq:Oalpha10} and \eqref{eq:Oalpha_01alpha2}, we conclude that
\begin{multline*}
h^{\pare{1}}_{tt} +\Lambda  h^{\pare{1}}+\beta\Lambda^3 h^{\pare{1}}+\alpha^0_1 h^{\pare{1}}_t-\alpha_1^0 \alpha_2\partial_{1}^2 h^{\pare{1}}-\alpha_2\partial_{1}^2 h^{\pare{1}}_t\\
=-\Lambda\left(\left(\cH h^{(0)}_t\right)^2\right)+\partial_1\comm{\cH}{h^{(0)}}\Lambda h^{(0)}+\beta\partial_1\comm{\cH}{h^{(0)}}\Lambda^3 h^{\pare{0}}
+\alpha_2\partial_1\comm{\cH}{\cH h^{\pare{0}} _t}\cH \partial_1 ^2 h^{\pare{0}}\\+\alpha_2\Lambda\left(\cH h^{\pare{0}} _t\cH \partial_1 ^2 h^{\pare{0}}\right)-\alpha_2^2\partial_1\comm{\cH}{\partial_{1}^2 h^{\pare{0}}}\partial_{1}^2 h^{\pare{0}}.
\end{multline*}
We define the renormalized variable
\begin{equation}\label{renormalized}
f =  h^{\pare{0}} + \varepsilon h^{\pare{1}}.
\end{equation}
Using
$$
\varepsilon h^{\pare{0}} =\varepsilon f +\mathcal{O}(\varepsilon^2),
$$
and neglecting errors $\mathcal{O}(\varepsilon^2)$, we conclude the following model:
\begin{multline}\label{models0}
f_{tt} +\Lambda  f+\beta \Lambda^3f+\alpha^0_1 f_t-\alpha_1^0 \alpha_2\partial_{1}^2 f-\alpha_2\partial_{1}^2 f_t\\
=\varepsilon\bigg{[}-\Lambda\left(\left(\cH f_t\right)^2\right)+\partial_1\comm{\cH}{f}\Lambda f+\beta\partial_1\comm{\cH}{f}\Lambda^3 f
+\alpha_2\partial_1\comm{\cH}{\cH f_t}\cH \partial_1 ^2 f\\+\alpha_2\Lambda\left(\cH f_t\cH \partial_1 ^2 f\right)-\alpha_2^2\partial_1\comm{\cH}{\partial_{1}^2 f}\partial_{1}^2 f\bigg{]}.
\end{multline}
When $\alpha_2=0$, equation \eqref{models0} is an asymptotic model of the damped water waves system proposed by Jiang, Ting, Perlin \& Schultz \cite{jiang1996moderate} and Wu, Liu \& Yue \cite{wu2006note}. Also, when $\alpha_2=\alpha_1^0=0$, equation \eqref{models0} recovers the quadratic $h-$model in \cite{CGSW18, matsuno1992nonlinear, matsuno1993two, matsuno1993nonlinear, AkMi2010, AkNi2010}.

\section{The asymptotic model for damped water waves when $s=2$} \label{ref:asymdampeds2}
In this section we focus on the case $s=2$ (the model by Dias, Dyachenko, and Zakharov \cite{dias2008theory}). In this case we have that
$$
\mathcal{D}^2\Phi=\partial_2^2\Phi.
$$
We use the ansatz \eqref{eq:ansatz} and follow the previous steps. The first term in the series solves
\begin{subequations}\label{eq:n0s2}
\begin{align}
\Delta \Phi^{(0)} &=0 ,  &&\text{in } \Omega\times[0,T]\,,\\
\Phi^{(0)}  &= \xi^{(0)}  \qquad &&\text{on }\Gamma\times[0,T],\\
\xi^{(0)} _t &=- h^{\pare{0}} +\beta\partial_{1}^2h^{\pare{0}} -\alpha_1^{2}\partial_2^2 \Phi^{(0)}  &&\text{ on }\Gamma\times[0,T],\\
h_t^{(0)} &=\partial_2\Phi^{(0)} +\alpha_2\partial_{1}^2 h^{\pare{0}} &&\text{ on }\Gamma\times[0,T].
\end{align}
\end{subequations}
Taking a time derivative of the equation (\ref{eq:n0s2}d), using the fact that
$$
\partial_2\Phi^{(0)}\bigg{|}_{x_2=0}=\Lambda \xi^{(0)}=h^{(0)}_t-\alpha_2\partial_1^2h^{(0)}
$$
and substituting (\ref{eq:n0s2}c), we find that
\begin{align*}
h_{tt}^{(0)} &=\Lambda\left(- h^{\pare{0}} +\beta\partial_{1}^2h^{\pare{0}} -\alpha_1^{2}\partial_2^2\Phi^{(0)}\right) +\alpha_2\partial_{1}^2 h^{\pare{0}}_t &&\text{ on }\Gamma\times[0,T].
\end{align*}
Similarly, due to the fact that
$$
\partial_2^2\Phi^{(0)}\bigg{|}_{x_2=0}=\Lambda^2\xi^{(0)}=\Lambda\left[h^{(0)}_t-\alpha_2\partial_1^2h^{(0)}\right]
$$
we find that the previous equation for $h_{tt}$ can be written as
\begin{align}\label{eq:eq_h0_1s2}
h_{tt}^{(0)} &=- \Lambda h^{\pare{0}} -\beta\Lambda^3h^{\pare{0}} +\alpha_1^{2}\partial_1^2 h^{\pare{0}} _t -\alpha_1^2\alpha_2\partial_{1}^4 h^{\pare{0}} +\alpha_2\partial_{1}^2 h^{\pare{0}}_t &&\text{ on }\Gamma\times[0,T].
\end{align}

Analogously as in \eqref{eq:n1}, for $n=1$, we find that
\begin{subequations}\label{eq:n1s2}
\begin{align}
\Delta \Phi^{(1)} &=\partial_1^2 h^{\pare{0}} \ \partial_2 \Phi^{(0)} + 2\partial_1 h^{\pare{0}} \ \partial_{12}\Phi^{(0)}  ,  &&\text{in } \Omega\times[0,T]\,,\\
\Phi^{(1)}  &= \xi^{(1)} \qquad &&\text{on }\Gamma\times[0,T],\\
\xi^{(1)}_t &=\frac{1}{2}\left[(\partial_2\Phi^{(0)})^2-(\partial_1\Phi^{(0)})^2\right]&&\nonumber\\
&\quad- h^{\pare{1}}+\beta\partial_{1}^2h^{\pare{1}}-\alpha_1^{2}\partial_2^2\Phi^{(1)} +\alpha_2\partial_2\Phi^{(0)} \partial_1^2 h^{\pare{0}}&&\text{ on }\Gamma\times[0,T],\\
h^{\pare{1}}_t&=-\partial_1 h^{\pare{0}} \partial_1 \Phi^{(0)} +\partial_2\Phi^{(1)}+\alpha_2\partial_{1}^2 h^{\pare{1}}&&\text{ on }\Gamma\times[0,T],
\end{align}
\end{subequations}

We use Lemma \ref{lem:solutions_Poisson} and \eqref{eq:explicit_expression_Phi1} to find that
\begin{align*}
\partial_2 \Phi^{\pare{1}}\bigg{|}_{x_2=0} &=   \Lambda \xi^{\pare{1}} -\comm{\Lambda}{h^{\pare{0}}}\Lambda \xi^{ \pare{0}}\\
\partial_2^2 \Phi^{\pare{1}} \bigg{|}_{x_2=0} & = \Lambda^2\xi^{(1)}+\partial_1^2 h^{\pare{0}}\Lambda\xi^{\pare{0}} + 2\partial_1 h^{\pare{0}}\partial_1 \Lambda \xi^{\pare{0}}.
\end{align*}

Then we find the following system of equations
\begin{align}
h^{\pare{1}}_t & =-\partial_1 h^{\pare{0}} \partial_1 \xi^{(0)} +   \Lambda \xi^{\pare{1}} -\comm{\Lambda }{  h^{\pare{0}}}\Lambda \xi^{ \pare{0}}+\alpha_2\partial_{1}^2 h^{\pare{1}}, \label{eq:n1h1s=2} \\
\xi^{(1)}_t &=\frac{1}{2}\left[ \pare{ \Lambda\xi^{(0)}}^2 - \pare{\partial_1\xi^{(0)}}^2\right]  \nonumber  \\
&\quad- h^{\pare{1}}+\beta\partial_{1}^2h^{\pare{1}}-\alpha_1^{2}\left[\Lambda^2\xi^{(1)}+\partial_1^2 h^{\pare{0}}\Lambda\xi^{\pare{0}} + 2\partial_1 h^{\pare{0}}\partial_1 \Lambda \xi^{\pare{0}}\right]     
+\alpha_2 \Lambda\xi^{(0)}\partial_1^2 h^{\pare{0}}. \label{eq:n1xi1s=2}
\end{align}
These equations are the analog (when $s=2$) of the equations \eqref{eq:n1h1} and \eqref{eq:n1xi1}.

As before, we want to reduce everything to a single equation for $h^{(1)}$ and $h^{(0)}$. Using (\ref{eq:n0s2}d), we find that
\begin{align*}
\Lambda \xi^{\pare{1}}&=h^{\pare{1}}_t-\partial_1 h^{\pare{0}} \cH \left[h^{(0)}_t-\alpha_2\partial_1^2h^{(0)}\right]-\alpha_2\partial_{1}^2 h^{\pare{1}}+\comm{\Lambda}{h^{\pare{0}}}\left[h^{(0)}_t-\alpha_2\partial_1^2h^{(0)}\right].
\end{align*}
As a consequence, we have that
\begin{equation*}
\begin{aligned}
\left.\alpha_1^2 \partial_2^2 \Phi^{\pare{1}}\right|_{x_2=0}& = \alpha_1^2\Lambda\set{ h^{\pare{1}}_t -\partial_1 h^{\pare{0}} \cH\bra{h^{\pare{0}} _t -\alpha_2\partial_{1}^2 h^{\pare{0}}} + \comm{\Lambda}{h^{\pare{0}}} \bra{h^{\pare{0}} _t -\alpha_2\partial_{1}^2 h^{\pare{0}}} - \alpha_2\partial_{1}^2 h^{\pare{1}} }\\
&\quad+\alpha_1^2\left(\partial_1^2 h^{\pare{0}}\bra{h^{\pare{0}} _t -\alpha_2\partial_{1}^2 h^{\pare{0}}}  + 2\partial_1 h^{\pare{0}}\partial_1 \bra{h^{\pare{0}} _t -\alpha_2\partial_{1}^2 h^{\pare{0}}} \right).
\end{aligned}
\end{equation*}

Time differentiating \eqref{eq:n1h1s=2} and inserting \eqref{eq:n1xi1s=2}, we deduce 
\begin{align*}
h^{\pare{1}}_{tt} & =-\partial_1 h^{\pare{0}}_t \partial_1 \xi^{(0)} +   \Lambda \xi^{\pare{1}}_t -\comm{\Lambda}{h^{\pare{0}}_t}\Lambda \xi^{ \pare{0}}+\alpha_2\partial_{1}^2 h^{\pare{1}}_t-\partial_1 h^{\pare{0}} \partial_1 \xi^{(0)}_t -\comm{\Lambda}{h^{\pare{0}}}\Lambda \xi^{ \pare{0}}_t, \\
&=\partial_1 h^{\pare{0}}_t \mathcal{H}\left[h^{(0)}_t-\alpha_2\partial_1^2h^{(0)}\right] -\comm{\Lambda}{h^{\pare{0}}_t}\left[h^{(0)}_t-\alpha_2\partial_1^2h^{(0)}\right]\\
&\quad+\alpha_2\partial_{1}^2 h^{\pare{1}}_t-\partial_1 h^{\pare{0}} \partial_1 \left[- h^{\pare{0}} +\beta\partial_{1}^2h^{\pare{0}} -\alpha_1^{2}\Lambda\left[h^{(0)}_t-\alpha_2\partial_1^2h^{(0)}\right]\right]\\
&\quad -\comm{\Lambda}{h^{\pare{0}}}\Lambda\left[- h^{\pare{0}} +\beta\partial_{1}^2h^{\pare{0}} -\alpha_1^{2}\Lambda\left[h^{(0)}_t-\alpha_2\partial_1^2h^{(0)}\right]\right],\\
&\quad+\frac{1}{2}\Lambda\left[ \pare{ \Lambda\xi^{(0)}}^2 - \pare{\partial_1\xi^{(0)}}^2\right]  \nonumber  \\
&\quad+\Lambda\left[- h^{\pare{1}}+\beta\partial_{1}^2h^{\pare{1}}-\alpha_1^{2}\left[\Lambda^2\xi^{(1)}+\partial_1^2 h^{\pare{0}}\Lambda\xi^{\pare{0}} + 2\partial_1 h^{\pare{0}}\partial_1 \Lambda \xi^{\pare{0}}\right]     
+\alpha_2 \Lambda\xi^{(0)}\partial_1^2 h^{\pare{0}}\right]\\
&=\partial_1 h^{\pare{0}}_t \mathcal{H}\left[h^{(0)}_t-\alpha_2\partial_1^2h^{(0)}\right] -\comm{\Lambda}{h^{\pare{0}}_t}\left[h^{(0)}_t-\alpha_2\partial_1^2h^{(0)}\right]\\
&\quad+\alpha_2\partial_{1}^2 h^{\pare{1}}_t-\partial_1 h^{\pare{0}} \partial_1 \left[- h^{\pare{0}} +\beta\partial_{1}^2h^{\pare{0}} -\alpha_1^{2}\Lambda\left[h^{(0)}_t-\alpha_2\partial_1^2h^{(0)}\right]\right]\\
&\quad -\comm{\Lambda}{h^{\pare{0}}}\Lambda\left[- h^{\pare{0}} +\beta\partial_{1}^2h^{\pare{0}} -\alpha_1^{2}\Lambda\left[h^{(0)}_t-\alpha_2\partial_1^2h^{(0)}\right]\right],\\
&\quad+\frac{1}{2}\Lambda\left[ \pare{h^{(0)}_t-\alpha_2\partial_1^2h^{(0)}}^2 - \pare{\mathcal{H}h^{(0)}_t-\alpha_2\mathcal{H}\partial_1^2h^{(0)}}^2\right]  \nonumber  \\
&\quad+\Lambda\bigg{\{}- h^{\pare{1}}+\beta\partial_{1}^2h^{\pare{1}}-\alpha_1^{2}\left[\partial_1^2 h^{\pare{0}}\left[h^{(0)}_t-\alpha_2\partial_1^2h^{(0)}\right] + 2\partial_1 h^{\pare{0}}\partial_1 \left[h^{(0)}_t-\alpha_2\partial_1^2h^{(0)}\right]\right]     \\
&\quad+\alpha_2 \left[h^{(0)}_t-\alpha_2\partial_1^2h^{(0)}\right]\partial_1^2 h^{\pare{0}}\bigg{\}}\\
&\quad-\alpha_1^{2}\Lambda^2\left\{h^{\pare{1}}_t-\partial_1 h^{\pare{0}} \cH \left[h^{(0)}_t-\alpha_2\partial_1^2h^{(0)}\right]-\alpha_2\partial_{1}^2 h^{\pare{1}}+\comm{\Lambda}{h^{\pare{0}}}\left[h^{(0)}_t-\alpha_2\partial_1^2h^{(0)}\right]\right\},
\end{align*}
where we have used the previous expression for $\Lambda\xi^{(1)}$. We group the different nonlinear contributions according to the coefficient in front: at $\mathcal{O}(1)$ we find \eqref{O(1)}, while at $\mathcal{O}(\beta)$ we have \eqref{O(beta)}. Using Tricomi identity \eqref{tricomi} to obtain
$$
\frac{1}{2}\Lambda\left[ \pare{h^{(0)}_t-\alpha_2\partial_1^2h^{(0)}}^2 - \pare{\mathcal{H}h^{(0)}_t-\alpha_2\mathcal{H}\partial_1^2h^{(0)}}^2\right]=\partial_1\left[ \pare{h^{(0)}_t-\alpha_2\partial_1^2h^{(0)}}\pare{\mathcal{H}h^{(0)}_t-\alpha_2\mathcal{H}\partial_1^2h^{(0)}}\right],
$$
we find that the $\mathcal{O}(\alpha_2)$ contribution is given by \eqref{eq:Oalpha_2v1} and, as a consequence, it can be further simplify to conclude \eqref{eq:Oalpha_2}. At $O(\alpha_2\alpha_2)$ we have the terms \eqref{eq:Oalpha_22}. We collect now the $\mathcal{O}(\alpha_1^2)$ terms:
\begin{multline}\label{eq:Oalpha12s2}
\alpha_1^2\bigg{[}\partial_1 h^{\pare{0}} \partial_1\Lambda h^{\pare{0}} _t-\partial_1^2\left(\partial_1 h^{\pare{0}} \cH h^{\pare{0}} _t\right)  - \comm{\Lambda}{h^{\pare{0}}} \partial_1^2 h^{\pare{0}} _t  + \partial_1^2\comm{\Lambda}{h^{\pare{0}}} h^{\pare{0}} _t\\
-\Lambda\left\{\partial_1^2 h^{\pare{0}}h^{(0)}_t+2\partial_1 h^{\pare{0}}\partial_1 h^{(0)}_t\right\}\bigg{]}=\alpha_1^2\bigg{[}-\partial_1^3 h^{\pare{0}} \cH h^{\pare{0}} _t-2\partial_1^2 h^{\pare{0}} \Lambda h^{\pare{0}} _t\\
 + 2\comm{\Lambda}{\partial_1h^{\pare{0}}} \partial_1 h^{\pare{0}} _t+\comm{\Lambda}{\partial_1^2 h^{\pare{0}}} h^{\pare{0}} _t-\Lambda\left\{\partial_1^2 h^{\pare{0}}h^{(0)}_t+2\partial_1 h^{\pare{0}}\partial_1 h^{(0)}_t\right\}\bigg{]}\\
    = -\alpha_1^2\partial_1\comm{\partial_1^2}{h^{\pare{0}}} \cH h^{\pare{0}} _t.
\end{multline}
Finally, we consider the $\mathcal{O}(\alpha_2\alpha_1^2)$ terms and obtain
\begin{multline}\label{eq:Oalpha_01alpha2s2}
\alpha_1^2\alpha_2\bigg{[}-\partial_1 h^{\pare{0}}\Lambda\partial_{1}^3 h^{\pare{0}}+ \comm{\Lambda}{h^{\pare{0}}}\partial_{1}^4 h^{\pare{0}}-\partial_1^2\comm{\Lambda}{h^{\pare{0}}}\partial_{1}^2 h^{\pare{0}} \\
+\Lambda\left[\left(\partial_1^2 h^{(0)}\right)^2+2\partial_1h^{(0)}\partial_1^3 h^{(0)}\right]+\partial_1^2\left(\partial_1 h^{(0)}\partial_1\Lambda h^{(0)}\right)
\bigg{]}\\
=\alpha_1^2\alpha_2\partial_1\comm{\partial_1^2}{h^{\pare{0}}}\Lambda\partial_{1} h^{\pare{0}}.
\end{multline}
Collecting \eqref{O(1)}, \eqref{O(beta)}, \eqref{eq:Oalpha_2}, \eqref{eq:Oalpha_22}, \eqref{eq:Oalpha12s2} and \eqref{eq:Oalpha_01alpha2s2}, we conclude the following equation for $h^{(1)}$
\begin{multline*}
h^{(1)}_{tt}-(\alpha_1^2+\alpha_2)\partial_{1}^2 h^{\pare{1}}_t+ \Lambda h^{\pare{1}}+\beta\Lambda^3 h^{\pare{1}}+\alpha_1^{2}\alpha_2\partial_{1}^4 h^{\pare{1}}\\
=-\Lambda\left(\left(\cH h^{(0)}_t\right)^2\right)+\partial_1\comm{\cH}{h^{(0)}}\Lambda h^{(0)}+\beta\partial_1\comm{\cH}{h^{(0)}}\Lambda^3 h^{\pare{0}}
\\
+\alpha_2\partial_1\comm{\cH}{\cH h^{\pare{0}} _t}\cH \partial_1 ^2 h^{\pare{0}}+\alpha_2\Lambda\left(\cH h^{\pare{0}} _t\cH \partial_1 ^2 h^{\pare{0}}\right)
+\alpha_1^2\alpha_2\partial_1\comm{\partial_1^2}{h^{\pare{0}}}\Lambda\partial_{1} h^{\pare{0}}\\
-\alpha_1^2\partial_1\comm{\partial_1^2}{h^{\pare{0}}} \cH h^{\pare{0}} _t
-\alpha_2\alpha_2\partial_1\comm{\cH}{\partial_{1}^2 h^{\pare{0}}}\partial_{1}^2 h^{\pare{0}}. 
\end{multline*}
Thus, neglecting errors of order $\mathcal{O}(\varepsilon^2)$, we conclude the following model for the renormalized variable \eqref{renormalized}:
\begin{multline}\label{models2}
f_{tt}-(\alpha_1^2+\alpha_2)\partial_{1}^2 f_t+ \Lambda f+\beta\Lambda^3 f+\alpha_1^{2}\alpha_2\partial_{1}^4 f\\
= \varepsilon\bigg\lbrace-\Lambda\left(\left(\cH f_t\right)^2\right)+\partial_1\comm{\cH}{f}\Lambda f +\beta\partial_1\comm{\cH}{f}\Lambda^3 f
\\
+\alpha_2\partial_1\comm{\cH}{\cH f _t}\cH \partial_1 ^2 f+\alpha_2\Lambda\left(\cH f _t\cH \partial_1 ^2 f \right)
+\alpha_1^2\alpha_2\partial_1\comm{\partial_1^2}{f}\Lambda\partial_{1} f\\
-\alpha_1^2\partial_1\comm{\partial_1^2}{f} \cH f _t
-\alpha_2\alpha_2\partial_1\comm{\cH}{\partial_{1}^2 f}\partial_{1}^2 f \bigg\rbrace. 
\end{multline}
When $\alpha_2=\alpha_1^2$, equation \eqref{models2} is an asymptotic model of the damped water waves system proposed by Dias, Dyachenko, and Zakharov \cite{dias2008theory}. Also, when $\alpha_2=\alpha_1^2=0$, equation \eqref{models2} again recovers the quadratic $h-$model in \cite{CGSW18, matsuno1992nonlinear, matsuno1993two, matsuno1993nonlinear, AkMi2010, AkNi2010}.

\section{Craig-Sulem models for damped water waves} \label{ref:asymdampeds3}
The pioneer work of Craig \& Sulem \cite{craig1993numerical} (see also \cite{milder1991improved,milder1992improved,CraigGS}) lead, among other things, to several asymptotic models obtained by truncating a Taylor series for the Dirichlet-to-Neumann operator present in the Zakharov formulation of the water waves problem \cite{zakharov1968stability}. Probably the most famous model of this type is the Craig-Sulem WW2 (see \cite{ambrose2014ill,CGSW18,liu2019sufficiently}):
\begin{align}
f_t & =-\varepsilon\partial_1 f \partial_1 \zeta +   \Lambda \zeta -\varepsilon\comm{\Lambda}{f}\Lambda \zeta, \label{eq:CSww21} \\
\zeta_t &=\frac{\varepsilon}{2}\left[ \pare{ \Lambda\zeta}^2 - \pare{\partial_1\zeta}^2\right]  - f+\beta\partial_{1}^2f. \label{eq:CSww22}
\end{align}
Using Tricomi identity \eqref{tricomi},
$$
\pare{ \Lambda\zeta}^2 - \pare{\partial_1\zeta}^2=2\cH\left(\partial_1 f\Lambda f\right),
$$
so the previous system can be equivalently written as
\begin{align}
f_t & =-\varepsilon\partial_1 f \partial_1 \zeta +   \Lambda \zeta -\varepsilon\comm{\Lambda}{f}\Lambda \zeta, \label{eq:CSww21v2} \\
\zeta_t &=\varepsilon\cH\left(\partial_1 f\Lambda f\right)  - f+\beta\partial_{1}^2f. \label{eq:CSww22v2}
\end{align}
\subsection{Case $s=0$}
Using \eqref{eq:n1h1} and \eqref{eq:n1xi1} we find that, up to an error $\mathcal{O}(\varepsilon^2)$ the variables 
\begin{equation}\label{newvariables}
f=h^{(0)}+\varepsilon h^{(1)},\;\zeta=\xi^{(0)}+\varepsilon \xi^{(1)},
\end{equation}
solve the system
\begin{align}
f_t & =-\varepsilon\partial_1 f \partial_1 \zeta +   \Lambda \zeta -\varepsilon\comm{\Lambda}{f}\Lambda \zeta+\alpha_2\partial_{1}^2 f, \label{eq:n1h1v2} \\
\zeta_t &=\varepsilon\cH\left(\partial_1 f\Lambda f\right)  - f+\beta\partial_{1}^2f-\alpha_1^{0}\zeta     
+\alpha_2 \varepsilon\Lambda\zeta\partial_1^2 f. \label{eq:n1xi1v2}
\end{align}
\subsection{Case $s=2$}
Using \eqref{eq:n1h1s=2} and \eqref{eq:n1xi1s=2}, we also find the viscous analog (called Craig-Sulem WWV2 \cite{kakleas2010numerical,ambrose2012well}) of the Craig-Sulem WW2 model corresponding for the model of Dias, Dyachenko, and Zakharov \cite{dias2008theory} of water waves with viscosity
\begin{align}
f_t & =-\varepsilon\partial_1 f \partial_1 \zeta +   \Lambda \zeta -\varepsilon\comm{\Lambda}{f}\Lambda \zeta+\alpha_2\partial_{1}^2 f, \label{eq:CSdww1} \\
\zeta_t &=\varepsilon\cH\left(\partial_1 f\Lambda f\right) 
- f+\beta\partial_{1}^2f-\alpha_1^2\left( \Lambda^2\zeta+\varepsilon\partial_1^2 f\Lambda\zeta + 2\varepsilon\partial_1 f\partial_1 \Lambda \zeta\right)
+\alpha_2 \varepsilon \Lambda\zeta\partial_1^2 f. \label{eq:CSdww2}
\end{align}

\section{Study of the models and discussion} In this paper we have obtained a number of new models for damped water waves. Of course, one may ask why viscosity effects are required when studying water waves. Besides the fact that every liquid is viscous, there are a number of scenarios where the viscous damping needs to be taken into account. For instance, damping has been used to study standing surface waves generated in a vertically oscillating container (these waves are called Faraday waves) or the question of stabilization of the Benjamin-Feir stability \cite{wu2006note}.

In particular, we derived two nonlocal wave equations, namely,
\begin{multline}\label{models0B}
f_{tt} +\Lambda  f+\beta \Lambda^3f+\alpha^0_1 f_t-\alpha_1^0 \alpha_2\partial_{1}^2 f-\alpha_2\partial_{1}^2 f_t\\
=\varepsilon\bigg{[}-\Lambda\left(\left(\cH f_t\right)^2\right)+\partial_1\comm{\cH}{f}\Lambda f+\beta\partial_1\comm{\cH}{f}\Lambda^3 f
+\alpha_2\partial_1\comm{\cH}{\cH f_t}\cH \partial_1 ^2 f\\+\alpha_2\Lambda\left(\cH f_t\cH \partial_1 ^2 f\right)-\alpha_2^2\partial_1\comm{\cH}{\partial_{1}^2 f}\partial_{1}^2 f\bigg{]},
\end{multline}
and
\begin{multline}\label{models2B}
f_{tt}-(\alpha_1^2+\alpha_2)\partial_{1}^2 f_t+ \Lambda f+\beta\Lambda^3 f+\alpha_1^{2}\alpha_2\partial_{1}^4 f\\
= \varepsilon\bigg\lbrace-\Lambda\left(\left(\cH f_t\right)^2\right)+\partial_1\comm{\cH}{f}\Lambda f +\beta\partial_1\comm{\cH}{f}\Lambda^3 f
\\
+\alpha_2\partial_1\comm{\cH}{\cH f _t}\cH \partial_1 ^2 f+\alpha_2\Lambda\left(\cH f _t\cH \partial_1 ^2 f \right)
+\alpha_1^2\alpha_2\partial_1\comm{\partial_1^2}{f}\Lambda\partial_{1} f\\
-\alpha_1^2\partial_1\comm{\partial_1^2}{f} \cH f _t
-\alpha_2\alpha_2\partial_1\comm{\cH}{\partial_{1}^2 f}\partial_{1}^2 f \bigg\rbrace. 
\end{multline}
Equation \eqref{models0B} is an asymptotic model of the damped water waves system proposed by Jiang, Ting, Perlin \& Schultz \cite{jiang1996moderate} and Wu, Liu \& Yue \cite{wu2006note}, while equation \eqref{models2B} is an asymptotic model of the water waves with viscosity system proposed by Dias, Dyachenko, and Zakharov \cite{dias2008theory}. 

It is a natural question to ask whether these ideas can be extended to three dimensional waves. Although the extension would not be trivial, these ideas can be applied to three dimensions. This should be addressed in a future work.

Another reasonable is which model is better for which application. In general, it is assumed in the literature that \eqref{eq:all} is more realistic when $s=2$, regardless of whether $\delta_2=0$ or not (see \cite{wu2006note} for instance). That would mean that \eqref{models2B} corresponds to a more realistic description of viscous damping of water waves. 

One of the advantages of having an asymptotic model akin to \eqref{models0B} or \eqref{models2B} is that, as there is no Dirichlet-Neumann operator nor elliptic problem involved, it is easier and cheaper to simulate than the full problem \eqref{eq:all}. However, when $\alpha_1^2,\alpha_2\neq0$, the presence of higher order operators as the bilaplacian may cause numerical difficulties. Thus, although \eqref{models2B} is linked to a more realistic description, its implementation may not be straightforward. A careful numerical study of these models should be addressed elsewhere. Also, this numerical study could help to make the decission of which models is better for which application.

\subsection{Typical values of the dimensionless parameters}
Let's consider a numerical example. The value of the physical parameters is (see \cite{Lannes13}):
$$
G=9.8m/s^2,\;\gamma=72\cdot 10^{-3} kg/s^2,\;\rho=1029 kg/m^3.
$$
We consider a wave of size
$$
H=0.02m,\;\;L=0.6m.
$$
This wave follows the scenario in \cite{jiang1996moderate}. Recalling \eqref{eq:dimensionless_parameters} (where $\delta_2=\nu$), we have that
\begin{align} \label{eq:dimensionless_parameters2}
\varepsilon\approx0.03, && \beta=\frac{72\cdot 10^{-3}}{1029 \cdot 9.8\cdot (0.6)^2}\approx 2\cdot 10^{-5},
\end{align}
According to \cite[Section 4]{jiang1996moderate}, the experimental decay rate in the scenario modelled by \eqref{models0B} is estimated as $0.05s^{-1}$. Also, following \cite{dias2008theory} we have that the right viscosity to be used in these applications is the eddy viscosity value 
$$
\nu=10^{-3}.
$$ 
That means that
\begin{align} \label{eq:dimensionless_parameters3}
\alpha_1^{0}=\frac{0.05}{\sqrt{9.8}\cdot (0.6)^{-\frac{1}{2}}}\approx0.01, && \alpha_1^2=\alpha_2=\frac{10^{-3}}{\sqrt{9.8}(0.6)^{3/2}}\approx 6.8\cdot 10^{-4}.
\end{align}
Then, we see that viscous damping effects are at the same level as $\varepsilon^2$ and are somehow more relevant than surface tension effects.

\subsection{Linear analysis and dispersion relations}
In this section we are going to study the dispersion relation of the models (see also \cite{dutykh2009visco2}). In the case where viscous effects are neglected ($\alpha_1^s=\alpha_2=0$) the model was studied in \cite{CGSW18}. In this case, the dispersion relation is
\begin{align}\label{eq:DR_inviscid}
\omega_I\pare{k} = \sqrt{\av{k}\pare{1+\beta\av{k}^2}}, && \beta \geqslant 0,
\end{align}
which is, of course, the same dispersion relation as for the full water waves problem with infinite depth.

We want to understand now how this dispersion relation is affected by the viscous effects. Keeping only the linear terms in the equations \eqref{models0B} and \eqref{models2B} and inserting the standard plane wave ansatz
$$
f(x,t)=e^{ikx-i\omega t}, 
$$
we obtain the following dispersion relations
\begin{align}
\omega_{\pm}^{\bra{0}}\pare{k} &= \pm \frac{\sqrt{-(|k|^2\alpha_2+\alpha_1^0)^2+4(|k|(1+\beta|k|^2)+|k|^2\alpha_1^0\alpha_2)}}{2}\nonumber\\
&\quad-\frac{i\pare{\alpha_1^0 + \alpha_2\av{k}^2}}{2}, \label{eq:DR_0}\\
\omega_{\pm}^{\bra{2}}\pare{k} &=  \pm \frac{\sqrt{-\pare{\alpha_1^2 + \alpha_2}^2\av{k}^4+4(|k|(1+\beta|k|^2)+\alpha_1^2\alpha_2|k|^4)}}{2}\nonumber\\
&\quad -\frac{i\pare{\alpha_1^2 + \alpha_2}\av{k}^2}{2}, \label{eq:DR_2}
\end{align}
where $\omega_{\pm}^{\bra{0}}$ and $\omega_{\pm}^{\bra{2}}$ correspond to equation \eqref{models0B} and \eqref{models2B}, respectively. These dispersion relations are valid for the whole range of values of the dimensionless parameters. We emphasize that the imaginary parts present in the previous expressions for the dispersion relations imply parabolic behavior or, if $\alpha_2=0$ in \eqref{eq:DR_0}, at least absortion.

Using the previous numerical values \eqref{eq:dimensionless_parameters2} and \eqref{eq:dimensionless_parameters3}, we find that, neglecting terms of order $O(10^{-5})$ for a large range of $k'$s, the dispersion relation \eqref{eq:DR_0} can be approximated by
\begin{align*}
\omega_{\pm}^{\bra{0}}\pare{k} &\approx \pm \frac{\sqrt{-(\alpha_1^0)^2+4|k|}}{2}-\frac{i\pare{\alpha_1^0 + \alpha_2\av{k}^2}}{2}.
\end{align*}
Similarly,
\begin{align*}
\omega_{\pm}^{\bra{2}}\pare{k} &\approx  \pm \sqrt{|k|} -\frac{i\pare{\alpha_1^2 + \alpha_2}\av{k}^2}{2}, 
\end{align*}
From the previous dispersion relations, $\omega_{\pm}^{\bra{s}}$, and the dispersion relation for the inviscid model $\omega_I$, we see that both models \eqref{models0B} and \eqref{models2B} have a parabolic behavior. In fact, the dissipation rate in \eqref{models0B} when $\alpha_2=0$ is independent of the Fourier mode $k$, while, for model \eqref{models2B} the dissipation is purely of parabolic type $O(|k|^2)$ (see \cite{lamb1932hydrodynamics}).

\appendix
\section{The explicit solution of an elliptic problem}
\begin{lemma}\label{lem:solutions_Poisson}
Let us consider the Poisson equation
\begin{equation}
\label{eq:Poisson}
\left\lbrace
\begin{aligned}
& \Delta u \pare{x_1 , x_2} &&= b \pare{x_1 , x_2}, & \pare{x_1, x_2} &\in \mathbb{S}^1\times  \pare{-\infty, 0}, \\
&  u \pare{x_1, 0} &&= g \pare{x_1} , & x_1 &\in \mathbb{S}^1, \\
& \lim_{x_2\rightarrow -\infty}\partial_2 u \pare{x_1, x_2} && =0, & x_1 &\in \mathbb{S}^1,
\end{aligned}
\right. 
\end{equation}
where we  assume that the forcing $ b$ and the boundary data $g$ are smooth and decay sufficiently fast at infinity. Then, the unique solution $ u $ of \eqref{eq:Poisson} is given by
\begin{equation} \label{eq:solution_Poisson}
\begin{aligned}
u \pare{x_1, x_2} =  {-} \frac{1}{\sqrt{2\pi}} \sum_{k=-\infty}^{\infty}& \left\lbrace  \frac{1}{\av{k}} \bra{\frac{1}{2} \int _{ -\infty}^0  \hat{b}\pare{k, y_2} e^{\av{k} y_2 } \textnormal{d} y_2  {-} \av{k} \hat{g}\pare{k} } e^{\av{k} x_2} \right.  \\
& -\frac{1}{2\av{k}} \int _{ -\infty}^0  \hat{b}\pare{k, y_2} e^{\av{k} y_2 } \textnormal{d} y_2
\  e^{- \av{k} x_2}  \\
& \left. 
+ \int_0^{x_2} \frac{\hat{b}\pare{k, y_2}}{2\av{k}} \bra{e^{\av{k}\pare{y_2 - x_2 }} - e^{\av{k}\pare{x_2 - y_2 }}} \textnormal{d} y_2
\right\rbrace e^{ik x_1} , 
\end{aligned}
\end{equation} 
where the operator $ \hat{\cdot} $ denotes the Fourier transform in the variable $ x_1 $. In particular
\begin{align}
\partial_2u \pare{x_1, 0} & =\int _{ -\infty}^0  e^{ y_2 \Lambda}b\pare{x_1, y_2} \textnormal{d} y_2+\Lambda g(x_1),\label{eq:pa2u0}\\
\partial_2^2 u \pare{x_1, 0} & =-\partial_1^2 g(x_1)+b(x_1,0).\label{eq:pa22u0}
\end{align}
\end{lemma}

\begin{proof}
Let us apply the Fourier transform to the equation \eqref{eq:Poisson}, this transforms the PDE \eqref{eq:Poisson} in the following series of second-order inhomogeneous costant coefficients ODE's
\begin{equation}
\label{eq:Poisson_Fourier}
\left\lbrace
\begin{aligned}
& -k^2 \hat{u}\pare{k, x_2} + \partial_2^2 \hat{u}\pare{k, x_2} = \hat{b} \pare{k, x_2}, & \pare{k, x_2} & \in \mathbb{Z}\times \pare{-\infty, 0}, \\
&\hat{u} \pare{k, 0} = \hat{g}\pare{k}, & k & \in \mathbb{Z} , \\
& \lim_{x_2\rightarrow -\infty}\partial_2 \hat{u}\pare{k, x_2} =0 , & k & \in \mathbb{Z}. 
\end{aligned}
\right. 
\end{equation}

The generic solution of \eqref{eq:Poisson_Fourier} can be deduced using the variation of parameters method, whence 
\begin{equation}\label{eq:generic_solution_ODE}
\hat{u}\pare{k, x_2} = C_1 \pare{k} e^{\av{k}x_2} + C_2 \pare{k} e^{- \av{k}x_2} - \int_0^{x_2} \frac{\hat{b}\pare{k, y_2}}{2\av{k}} \bra{e^{\av{k}\pare{y_2 - x_2 }} - e^{\av{k}\pare{x_2 - y_2 }}} \textnormal{d} y_2.
\end{equation}

The boundary conditions determine the values of the $ C_i $'s:
\begin{align}\label{eq:Ci}
 C_2 \pare{k} & = - \frac{1}{2\av{k}} \int _{ -\infty}^0 \hat{b}\pare{k, y_2} e^{\av{k} y_2 } \textnormal{d} y_2, 
 &
 C_1\pare{k} & = - C_2\pare{k} + \hat{g}\pare{k} .
\end{align}

We provide now the detailed computations for the sake of clarity. From the generic solution \eqref{eq:generic_solution_ODE} we easily derive that $ C_1 = -C_2 +\hat{g} $ simply setting $ x_2 = 0 $ and solving the resulting equation in $ C_1 $. Next we compute $ \partial_2 u $, which gives
\begin{multline}\label{eq:pa2uODE}
\partial_2 \hat{u}\pare{k, x_2} = \pare{\big.  - C_2\pare{k} + \hat{g}\pare{k}} \av{k} e^{\av{k}x_2} - C_2 \pare{k} \av{k} e^{- \av{k}x_2} \\
 + \frac{1}{2}\int_0^{x_2} {\hat{b}\pare{k, y_2}} \bra{e^{\av{k}\pare{y_2 - x_2 }} + e^{\av{k}\pare{x_2 - y_2 }}} \textnormal{d} y_2 . 
\end{multline} 
Due to the negative weight on the exponential, we deduce that
\begin{equation*}
\begin{aligned}
\lim _{x_2 \to -\infty} \bra{\pare{\big.  - C_2\pare{k} + \hat{g}\pare{k}} \av{k} e^{\av{k}x_2}} & =0 .
\end{aligned}
\end{equation*}
Let us now consider the limit
\begin{equation*}
\lim _{x_2 \to -\infty} \bra{ \frac{1}{2}\int_0^{x_2} {\hat{b}\pare{k, y_2}}  e^{\av{k}\pare{x_2 - y_2 }} \textnormal{d} y_2}  . 
\end{equation*}
We prove now that such limit is equal to zero by dominated convergence. Let us consider the family of functions
\begin{equation*}
\pare{f_{k, x_2}\pare{y_2}}_{x_2\in\bR_-} =  \pare{\big. 1_{\bra{x_2, 0}}\pare{y_2} \hat{b}\pare{k, y_2} e^{\av{k}\pare{x_2 - y_2}}}_{x_2\in\bR_-},
\end{equation*}
Since every element of such family is nonzero only when $ y_2 \in \bra{x_2, 0} $ we know that $ e^{\av{k}\pare{x_2 - y_2}} \leqslant 1 $, hence every $ f_{k, x_2} $ can be pointwise bounded by
\begin{equation*}
f_{k, x_2}\pare{y_2} \leqslant \av{\hat{b}\pare{k, y_2}}, 
\end{equation*}
uniformly in $ x_2 $. Moreover we assumed $ b\in L^2 \pare{\mathbb{S}^1 ; L^1 \pare{\bR_-}} $, hence for every $ k $ we have that $ \hat{b}\pare{k, \cdot} \in L^1 \pare{\bR_-} $ and we can indeed apply the Lebesgue dominated convergence theorem in order to deduce
\begin{equation*}
\lim _{x_2 \to -\infty} \bra{ \frac{1}{2}\int_0^{x_2} {\hat{b}\pare{k, y_2}}  e^{\av{k}\pare{x_2 - y_2 }} \textnormal{d} y_2} = 0 , 
\end{equation*}
for every $ k\in \mathbb{Z} $.
What remains is the following equality
\begin{equation*}
\lim _{x_2 \to -\infty} \bra{
-C_2\pare{k} \av{k}e^{-\av{k}x_2} +  \frac{1}{2}\int_0^{x_2} {\hat{b}\pare{k, y_2}} e^{\av{k}\pare{y_2 - x_2 }}  \textnormal{d} y_2 
} =0, 
\end{equation*}
which in turn gives the required constant
\begin{equation*}
C_2 \pare{k}  = - \frac{1}{2\av{k}} \int _{ -\infty}^0 \hat{b}\pare{k, y_2} e^{\av{k} y_2 } \textnormal{d} y_2. 
\end{equation*}
Setting $ x_2 =0 $ in \eqref{eq:pa2uODE} we find that
\begin{equation}\label{eq:pa2uODE_1}
\partial_2 \hat{u}\pare{k, 0} = -2\av{k}C_2 \pare{k} + \av{k} \hat{g}\pare{k}, 
\end{equation}
which reduces to \eqref{eq:pa2u0}. We now differentiate \eqref{eq:pa2uODE} in $ x_2 $ obtaining
\begin{multline}\label{eq:pa2uODE_2}
\partial_2^2 \hat{u}\pare{k, x_2} = \pare{\big.  - C_2\pare{k} + \hat{g}\pare{k}} \av{k}^2 e^{\av{k}x_2} + C_2 \pare{k} \av{k}^2 e^{- \av{k}x_2} \\
 + \hat{b}\pare{k, x_2} - \frac{\av{k}}{2}\int_0^{x_2} {\hat{b}\pare{k, y_2}} \bra{e^{\av{k}\pare{y_2 - x_2 }} - e^{\av{k}\pare{x_2 - y_2 }}} \textnormal{d} y_2 . 
\end{multline} 
Fixing $ x_2 =0 $ in \eqref{eq:pa2uODE_2} the previous equation simplifies to
\begin{equation*}
\partial_2^2 \hat{u} \pare{k, 0} = \av{k}^2 \hat{g}\pare{k}  + \hat{b}\pare{k, x_2} , 
\end{equation*}
which proves \eqref{eq:pa22u0}. 
\end{proof}

\section*{Acknowledgments}
The research of S.S. is supported by the Basque Government through the BERC 2018-2021 program and by Spanish Ministry of Economy and Competitiveness MINECO through BCAM Severo Ochoa excellence accreditation SEV-2017-0718 and through project MTM2017-82184-R funded by (AEI/FEDER, UE) and acronym "DESFLU". We thank the anonymous referees for their numerous suggestions that have improved the exposition of this article.

\begin{footnotesize}

\end{footnotesize}
\vspace{2cm}


\begin{thebibliography}{10}

\bibitem{AkMi2010}
Benjamin Akers and Paul~A Milewski.
\newblock Dynamics of three-dimensional gravity-capillary solitary waves in
  deep water.
\newblock {\em SIAM Journal on Applied Mathematics}, 70(7):2390--2408, 2010.

\bibitem{AkNi2010}
Benjamin Akers and David~P Nicholls.
\newblock Traveling waves in deep water with gravity and surface tension.
\newblock {\em SIAM Journal on Applied Mathematics}, 70(7):2373--2389, 2010.

\bibitem{ambrose2012well}
David Ambrose, Jerry Bona, and David Nicholls.
\newblock Well-posedness of a model for water waves with viscosity.
\newblock {\em Discrete Contin. Dyn. Syst. Ser. B}, 17:1113--1137, 2012.


\bibitem{ambrose2014ill}
David~M Ambrose, Jerry~L Bona, and David~P Nicholls.
\newblock On ill-posedness of truncated series models for water waves.
\newblock {\em Proceedings of the Royal Society A: Mathematical, Physical and
  Engineering Sciences}, 470(2166):20130849, 2014.

\bibitem{boussinesq1895lois}
J~Boussinesq.
\newblock Lois de l’extinction de la houle en haute mer.
\newblock {\em CR Acad. Sci. Paris}, 121(15-20):2, 1895.

\bibitem{CGSW18}
Arthur Cheng, Rafael Granero-Belinch\'on, Steve Shkoller, and Jon Wilkening.
\newblock {R}igorous {A}symptotic {M}odels of {W}ater {W}aves.
\newblock {\em Water Waves}, 1(1), 71-130, 2019.

\bibitem{Coutand-Shkoller:well-posedness-free-surface-incompressible}
Daniel Coutand and Steve Shkoller.
\newblock Well-posedness of the free-surface incompressible {E}uler equations
  with or without surface tension.
\newblock {\em J. Amer. Math. Soc.}, 20(3):829--930, 2007.

\bibitem{CraigGS}
Craig, W., Guyenne, P., \& Sulem, C.
\newblock Water waves over a random bottom
\newblock {\em Journal of Fluid Mechanics}, 640, 79--107, 2009.

\bibitem{craig1993numerical}
Walter Craig and Catherine Sulem.
\newblock Numerical simulation of gravity waves.
\newblock {\em Journal of Computational Physics}, 108(1):73--83, 1993.

\bibitem{dias2008theory}
Frederic Dias, Alexander~I Dyachenko, and Vladimir~E Zakharov.
\newblock Theory of weakly damped free-surface flows: a new formulation based
  on potential flow solutions.
\newblock {\em Physics Letters A}, 372(8):1297--1302, 2008.

\bibitem{dutykh2009visco}
Denys Dutykh.
\newblock Visco-potential free-surface flows and long wave modelling.
\newblock {\em European Journal of Mechanics-B/Fluids}, 28(3):430--443, 2009.

\bibitem{dutykh2009visco2}
Denys Dutykh.
\newblock Group and phase velocities in the free-surface visco-potential flow: new kind of boundary layer induced instability.
\newblock {\em Physics Letters A}, 373(36):3212--3216, 2009.

\bibitem{dutykh2007dissipative}
Denys Dutykh and Fr{\'e}d{\'e}ric Dias.
\newblock Dissipative boussinesq equations.
\newblock {\em Comptes Rendus Mecanique}, 335(9-10):559--583, 2007.

\bibitem{dutykhA}
Denys Dutykh and Olivier Goubet.
\newblock Derivation of dissipative Boussinesq equations using the Dirichlet-to-Neumann operator approach.
\newblock {\em Mathematics and Computers in Simulation}, 127, 80--93, 2016.

\bibitem{dutykhB}
Denys Dutykh.
\newblock Visco-potential free-surface flows and long wave modelling.
\newblock {\em European Journal of Mechanics-B/Fluids}, 28(3), 430-443, 2009.

\bibitem{dutykh2007viscous}
Denys Dutykh and Fr{\'e}d{\'e}ric Dias.
\newblock Viscous potential free-surface flows in a fluid layer of finite
  depth.
\newblock {\em Comptes Rendus Mathematique}, 345(2):113--118, 2007.

\bibitem{granero2019asymptotic}
Rafael Granero-Belinch\'on and Stefano Scrobogna.
\newblock Asymptotic models for free boundary flow in porous media.
\newblock {\em Physica D: Nonlinear Phenomena}, 392, 1-16, 2019.


\bibitem{granero2018asymptotic}
Rafael Granero-Belinch\'on and Stefano Scrobogna.
\newblock On an asymptotic model for free boundary {D}arcy flow in porous
  media.
\newblock Submitted, \url{https://arxiv.org/abs/1810.11798}.

\bibitem{GS}
Rafael Granero-Belinch\'on and Steve Shkoller.
\newblock A model for {R}ayleigh-{T}aylor mixing and interface turn-over.
\newblock {\em Multiscale Modeling and Simulation}, 15(1):274--308, 2017.

\bibitem{GP}
Guyenne, P., and Parau, E. I.
\newblock Numerical Simulation of Solitary-Wave Scattering and Damping in Fragmented Sea Ice
\newblock {\em The 27th International Ocean and Polar Engineering Conference}, 373--380, 2017.

\bibitem{hunt2014visco}
Matthew Hunt and Denys Dutykh.
\newblock Visco-potential flows in electrohydrodynamics.
\newblock {\em Physics Letters A}, 378(24-25):1721--1726, 2014.

\bibitem{jiang1996moderate}
Lei Jiang, Chao-Lung Ting, Marc Perlin, and William~W Schultz.
\newblock Moderate and steep faraday waves: instabilities, modulation and
  temporal asymmetries.
\newblock {\em Journal of Fluid Mechanics}, 329:275--307, 1996.

\bibitem{joseph2004dissipation}
Daniel~D Joseph and Jing Wang.
\newblock The dissipation approximation and viscous potential flow.
\newblock {\em Journal of Fluid Mechanics}, 505:365--377, 2004.

\bibitem{kakleas2010numerical}
Maria Kakleas and David~P Nicholls.
\newblock Numerical simulation of a weakly nonlinear model for water waves with
  viscosity.
\newblock {\em Journal of Scientific Computing}, 42(2):274--290, 2010.

\bibitem{kharif2010modulational}
C~Kharif, Roberto~Andr{\'e} Kraenkel, MA~Manna, and R~Thomas.
\newblock The modulational instability in deep water under the action of wind
  and dissipation.
\newblock {\em Journal of Fluid Mechanics}, 664:138--149, 2010.

\bibitem{khariff1996frequency}
C~Kharif, Skandrani C, and J~Poitevin.
\newblock the frequency down-shift phenomenon.
\newblock In {\em Mathematical Problems in the Theory of Water Waves: A
  Workshop on the Problems in the Theory of Nonlinear Hydrodynamic Waves, May
  15-19, 1995, Luminy, France}, volume 200, page 157. American Mathematical
  Soc., 1996.

\bibitem{lamb1932hydrodynamics}
H~Lamb.
\newblock {\em Hydrodynamics}.
\newblock Cambridge Univ Press,, 1932.

\bibitem{Lannes13}
David Lannes.
\newblock {\em The water waves problem}, volume 188 of {\em Mathematical
  Surveys and Monographs}.
\newblock American Mathematical Society, Providence, RI, 2013.
\newblock Mathematical analysis and asymptotics.

\bibitem{liu2019sufficiently}
Shunlian Liu and David~M Ambrose.
\newblock Sufficiently strong dispersion removes ill-posedness in truncated
  series models of water waves.
\newblock {\em Discrete \& Continuous Dynamical Systems-A}, 39(6):3123--3147,
  2019.

\bibitem{longuet1992theory}
Michael~S Longuet-Higgins.
\newblock Theory of weakly damped stokes waves: a new formulation and its
  physical interpretation.
\newblock {\em Journal of Fluid Mechanics}, 235:319--324, 1992.

\bibitem{matsuno1992nonlinear}
Y~Matsuno.
\newblock Nonlinear evolutions of surface gravity waves on fluid of finite
  depth.
\newblock {\em Physical review letters}, 69(4):609, 1992.

\bibitem{matsuno1993nonlinear}
Yoshimasa Matsuno.
\newblock Nonlinear evolution of surface gravity waves over an uneven bottom.
\newblock {\em Journal of fluid mechanics}, 249:121--133, 1993.

\bibitem{matsuno1993two}
Yoshimasa Matsuno.
\newblock Two-dimensional evolution of surface gravity waves on a fluid of
  arbitrary depth.
\newblock {\em Physical Review E}, 47(6):4593, 1993.

\bibitem{milder1991improved}
D~Michael Milder.
\newblock An improved formalism for wave scattering from rough surfaces.
\newblock {\em The Journal of the Acoustical Society of America},
  89(2):529--541, 1991.

\bibitem{milder1992improved}
D~Michael Milder and H~Thomas Sharp.
\newblock An improved formalism for rough-surface scattering. ii: Numerical
  trials in three dimensions.
\newblock {\em The Journal of the Acoustical Society of America},
  91(5):2620--2626, 1992.

\bibitem{ngom2018well}
Mari{\`e}me Ngom and David~P Nicholls.
\newblock Well-posedness and analyticity of solutions to a water wave problem
  with viscosity.
\newblock {\em Journal of Differential Equations}, 265(10):5031--5065, 2018.

\bibitem{ramani2019multiscale}
Raag Ramani and Steve Shkoller.
\newblock A multiscale model for rayleigh-taylor and richtmyer-meshkov
  instabilities.
\newblock {\em arXiv preprint arXiv:1904.04935}, 2019.

\bibitem{ruvinsky1991numerical}
KD~Ruvinsky, FI~Feldstein, and GI~Freidman.
\newblock Numerical simulations of the quasi-stationary stage of ripple
  excitation by steep gravity--capillary waves.
\newblock {\em Journal of Fluid Mechanics}, 230:339--353, 1991.

\bibitem{ruvinsky1985improvement}
KD~Ruvinsky and GI~Freidman.
\newblock Improvement of the first stokes method for the investigation of
  finite-amplitude potential gravity-capillary waves.
\newblock In {\em IX All-Union Symp. on Diffraction and Propagation Waves,
  Tbilisi: Theses of Reports}, volume~2, pages 22--25, 1985.

\bibitem{ruvinsky1987fine}
KD~Ruvinsky and GI~Freidman.
\newblock The fine structure of strong gravity-capillary waves.
\newblock {\em Nonlinear waves: Structures and Bifurcations, AV Gaponov-Grekhov
  and MI Rabinovich, eds. Moscow: Nauka}, pages 304--326, 1987.

\bibitem{Stokes_1847}
George~Gabriel Stokes.
\newblock {\em On the Theory of Oscillatory Waves}, volume~1 of {\em Cambridge
  Library Collection - Mathematics}.
\newblock Cambridge University Press, 1847.

\bibitem{touboul2010nonlinear}
Julien Touboul and C~Kharif.
\newblock Nonlinear evolution of the modulational instability under weak
  forcing and damping.
\newblock {\em Natural Hazards and Earth System Sciences}, 10(12):2589--2597,
  2010.

\bibitem{wang2006purely}
Jing Wang and Daniel~D Joseph.
\newblock Purely irrotational theories of the effect of the viscosity on the
  decay of free gravity waves.
\newblock {\em Journal of Fluid Mechanics}, 559:461--472, 2006.

\bibitem{wu2006note}
Guangyu Wu, Yuming Liu, and Dick~KP Yue.
\newblock A note on stabilizing the benjamin--feir instability.
\newblock {\em Journal of Fluid Mechanics}, 556:45--54, 2006.

\bibitem{wu1981long}
Theodore~Y Wu.
\newblock Long waves in ocean and coastal waters.
\newblock {\em Journal of Engineering Mechanics}, 107(EM3):501--522, 1981.

\bibitem{zabusky1971shallow}
NJ~Zabusky and CJ~Galvin.
\newblock Shallow-water waves, the korteweg-devries equation and solitons.
\newblock {\em Journal of Fluid Mechanics}, 47(4):811--824, 1971.

\bibitem{zakharov1968stability}
Vladimir~E Zakharov.
\newblock Stability of periodic waves of finite amplitude on the surface of a
  deep fluid.
\newblock {\em Journal of Applied Mechanics and Technical Physics},
  9(2):190--194, 1968.

\end{thebibliography}
\end{document}